\documentclass[12pt,reqno,twoside]{amsart}
\usepackage{amsthm,amsmath,amssymb,mathrsfs,amsbsy}
\usepackage{graphicx,epsfig,wrapfig,sidecap,subfig}
\usepackage[all]{xy}
\usepackage{proof}
\usepackage[usenames]{color}
\usepackage{hyperref}
\usepackage{bm}
\usepackage[capitalize]{cleveref}
\usepackage{listings}

\xyoption{dvips}
\xyoption{color}

\usepackage{color}

\title[On Visibility Problems]{On Visibility Problems with an Infinite Discrete, set of Obstacles}
\date{}
\author{Michael Boshernitzan}
\address{Rice University, Houston, TX, {\tt michael@math.rice.edu}}
\author{Yaar Solomon}
\address{Ben-Gurion University of the Negev, Beer-Sheva, Israel, {\tt yaars@bgu.ac.il}} 

\newcommand{\N}{{\mathbb{N}}}
\newcommand{\Z}{{\mathbb{Z}}}

\newcommand{\R}{{\mathbb{R}}}
\newcommand{\C}{{\mathbb{C}}}

\renewcommand{\SS}{{\mathbb{S}}}

\newcommand{\dist}{\mathbf{dist}}

\newcommand{\vis}{\mathbf{vis}}

\newcommand{\df}{{\, \stackrel{\mathrm{def}}{=}\, }}
\newcommand{\minus}{\!\setminus\!}

\newcommand{\eps}{\varepsilon}
\newcommand{\DTEV}{\delta$-$T$-$\varepsilon$-$V}
\newcommand{\DFEV}{\delta$-$F$-$\varepsilon$-$V}

\newcommand{\absolute}[1] {\left|{#1}\right|}
\newcommand{\ceil}[1]{\left\lceil{#1}\right\rceil}
\newcommand{\floor}[1]{\left\lfloor{#1}\right\rfloor}
\newcommand{\norm}[1]{\left\|{#1}\right\|}

\newcommand {\ignore}[1]  {}

\font\sn = cmssi8 scaled \magstep0
\font\si = cmti8 scaled \magstep0
\long\def\comyaar#1{\ifdraft{{\color{blue}\si Yaar:``#1'' }}\else\ignorespaces\fi}
\long\def\commichael#1{\ifdraft{\color{red}\sn MB:``#1'' }\else\ignorespaces\fi}

\theoremstyle{plain}
\newtheorem{thm}{Theorem}[section]
\newtheorem{lem}[thm]{Lemma}
\newtheorem{prop}[thm]{Proposition}

\theoremstyle{definition}
\newtheorem{definition}[thm]{Definition}
\newtheorem{question}[thm]{Question}

\newtheorem{remark}[thm]{Remark}

\numberwithin{equation}{section}
\swapnumbers

\newif\ifdraft\drafttrue
\usepackage{color}

\begin{document}

\ignore{This paper studies visibility problems in Euclidean spaces where the obstacles are points 
of an infinite discrete set. We prove several results that relate between the notion 
of visibility to the growth rate of the set of obstacles in $\R^d$. 
Our main result is the following. 
For every $\varepsilon>0$ and every relatively dense set $Y\subseteq\R^2$, 
there are points in the plane that are $\varepsilon$-hidden by $Y$. 
Namely, points that any ray initiated from them meets the 
$\varepsilon$-neighborhood of $Y$. 
This result generalizes a theorem of Dumitrescu and Jiang, 
which settled Mitchell's dark forest conjecture. \\}

\begin{abstract}
This paper studies visibility problems in Euclidean spaces $\R^d$ where the obstacles are the points
of infinite discrete sets  $Y\subseteq\R^d$.
A point  $x\in\R^d$  is called $\varepsilon$-visible for $Y$ (notation:  $x\in\vis(Y, \varepsilon))$  if there exists 
a ray  $L\subseteq\R^d$  emanating from $x$  such that  $\norm{y-z}\geq\varepsilon$,  for all $y\in Y\setminus\{x\}$  and  $z\in L$.
A point $x\in\R^d$  is called visible for $Y$   (notation:  $x\in\vis(Y))$  if  $x\in\vis(Y, \varepsilon))$, for some $\varepsilon>0$.\\
Our main result is the following. For every $\varepsilon>0$ and every relatively dense set $Y\subseteq\R^2$, 
$\vis(Y, \varepsilon))\neq\R^2$. This result generalizes a theorem of Dumitrescu and Jiang, 
which settled Mitchell's dark forest conjecture. On the other hand, we show that
there exists a relatively dense subset  $Y\subseteq \Z^d$  such that    $\vis(Y)=\R^d$.
(One easily verifies that $\vis(\Z^d)=\R^d\setminus\Z^d$, for all $d\geq 2$).
We derive a number of other results clarifying how the size of a sets  $Y\subseteq\R^d$
may affect the sets  $\vis(Y)$  and  $\vis(Y,\varepsilon)$.
We present a Ramsey type result concerning uniformly separated subsets of $\R^2$ whose growth is faster than linear.
\end{abstract}

\maketitle

\section{Introduction}\label{sec:introduction}

We use the following standard notations for a fixed integer $d\geq2$.  

For $x\in \R^d$,  we write  $\norm{x}$  to denote the Euclidean norm of $x$. 
Denote by  $\mathbb S^{d-1}= \{x\in\R^{d}\mid \norm x=1\}$   the unit sphere 
centered at the origin $\bf0$.
For two non-empty subsets  $A,B\subseteq \R^d$,  define
\begin{equation}\label{eq:dist}
\dist(A,B)=\inf\{\norm{a-b}\!: a\in A,\:  b\in B\}.
\end{equation}
Given $x\in \R^d$ and $v\in\mathbb{S}^{d-1}$, \emph{by the ray
from $x$ in direction $v$} we mean the set 
\[L_{x,v}=\{x+tv\mid t\in[0,\infty)\}\subseteq\R^{d}.\]
\smallskip

\begin{definition}\label{def:vis}
For a non-empty subset $Y\subseteq\R^{d}$, a direction $v\in\mathbb{S}^{d-1}$ and $\eps>0$,
define the following subsets of $\R^{d}$:
\begin{subequations}\label{eq:vis4}
\begin{align}
\vis(Y,v)\hspace{-8mm}& &&\df&&\hspace{-8mm}\{x\in \R^{d}\mid \dist(L_{x,v},Y\!\setminus\!\{x\})>0\},\\
\vis(Y)\hspace{-8mm}& &&\df&&\hspace{-8mm}\{x\in \R^{d}\mid x\in\vis(Y,v)\, \mbox{ for some }\, v\in\mathbb{S}^{d-1}\},\\
\vis(Y,v,\varepsilon)\hspace{-8mm}& &&\df&&\hspace{-8mm}\{x\in \R^{d}\mid \dist(L_{x,v},Y\!\setminus\!\{x\})\ge 
\varepsilon\},\\
\vis(Y;\varepsilon)\hspace{-8mm}& &&\df&&\hspace{-8mm}\{x\in \R^{d}\mid x\in\vis(Y,v,\varepsilon)\, \mbox{ for some }\, v\in\mathbb{S}^{d-1}\}.
\end{align}
\end{subequations}
\smallskip

If there is no ambiguity regarding the choice of  
$Y\subseteq\R^{d}$, then
\begin{itemize}\label{word description}
\item Points $x\in\vis(Y,v)$ are called \emph{visible from direction $v$}.
\item Points $x\in\vis(Y)$ are called \emph{visible}; \ points $x\in\R^d\!\setminus\!\vis(Y)$  
     are called \emph{hidden}.
\item Points $x\in\vis(Y,v,\varepsilon)$ are called \emph{$\eps$-visible from direction $v$}.
\item Points $x\in\vis(Y;\varepsilon)$ are called \emph{$\varepsilon$-visible}; \
         points $x\in\R^d\!\setminus\!\vis(Y;\varepsilon)$ are called \emph{$\varepsilon$-hidden}.
\end{itemize}

\noindent We add the specification "for $Y$" in the above word 
definitions if we want to indicate the dependence of these sets on $Y$.
\end{definition}

For $x\in\R^d$ and $r>0$, denote by $B(x,r)=\{y\in\R^d\mid \norm {y-x}<r\}$
the open  $r$-balls around $x$.
A set $Y\subseteq\R^d$ is called \emph{discrete} if the intersection of  $Y$  with every ball $B(r,x)$  is finite. 

By the growth rate of a discrete subset  $Y\subseteq\R^{d}$ we mean the integer valued function 
defined by the formula
\[
G_{Y}(r)=\#\big(\{y\in Y\mid \norm{y}< r\}\big) = \#\big(Y\cap  B({\bf0},r)\big) \qquad  (\text{for } \ r\geq0).
\]
(Here and henceforth  $\#S$ stands for the cardinality of a set $S$). 


A set $Y\subseteq\R^d$ is called \emph{relatively dense} if there is some $r>0$ so that every ball 
of radius $r$ intersects  $Y$,  i.\,e.\  $B(x,r)\cap Y\neq\varnothing$,  for  all $x\in\R^d$.
$Y$ is \emph{uni\-form\-ly separated} if there is some $\delta>0$ such that 
for every $y_1,y_2\in Y$ we have $\dist(y_1,y_2)\ge\delta$.  
We say $Y$ is $r$-dense, or $\delta$-separated, when we like to specify the constants $r$ and $\delta$. 

Note that if $Y\subseteq\R^d$ is  a uniformly separated set then 
\mbox{$\limsup_{r\to\infty}\limits\absolute{\frac{G_Y(r)}{r^d}}<\infty$}, and if  $Y\subseteq\R^d$ is discrete
and relatively dense then  $\limsup_{r\to\infty}\limits\absolute{\frac{r^d}{G_Y(r)}}<\infty$.

One of the objectives in this paper is to investigate how the growth rate of the set $Y$  
may be related to certain properties of the four sets defined in \eqref{eq:vis4}. Note that  $Y_1\subseteq Y_2$  implies the inequality
$
G_{Y_1}(r)\leq G_{Y_2}(r)
$
(for all $r>0$)  and the inclusions
\begin{align}
\vis(Y_2,v)&\subseteq\vis(Y_1,v),   &\vis(Y_2)&\subseteq\vis(Y_1),\label{eq:inclusion}\\
\vis(Y_2; \eps)&\subseteq\vis(Y_1; \eps),  &\vis(Y_2,v,\eps)&\subseteq\vis(Y_1,v, \eps).
\end{align}

In particular, we establish the following results for discrete subsets $Y\!\subseteq\R^2$: 

\begin{enumerate}
\item[1.]  If  $Y\!\subseteq\R^2$  is relatively dense then  $\vis(Y;\varepsilon)\neq\R^2$, for all  $\varepsilon>0$ 
(See Theorem \ref{thm:Michael's_conj}).
\item[2.]  On the other hand, there exists a relatively dense set  $Y\subseteq\R^d\: (d\ge 2)$ such that  $\vis(Y)=\R^d$  (See Theorem
\ref{thm:Every_point_is_visible}).
\item[3.]  If  $Y\subseteq\R^d\: (d\ge 2)$ is discrete and if  $G_Y(t)<\frac{t^{d-1}}{\log^{1+\varepsilon} t}$,  for some $\varepsilon>0$ and all large enough $t$,  then
$\vis(Y)=\R^d$ (See Theorem
\ref{thm:small_growth_implies_full_visibility_2}). 
\item[4.] On the other hand, we exhibit an uniformly separated set \ \text{$Y\subseteq\R^2$}  
  such that \ $\lim_{t\to\infty}\limits{\frac{G_Y(t)\log t}t}=1$ and  $\vis(Y)\!\!=\varnothing$
   (See Theorem \ref{thm:sublinear_growth_and_no_visibility}).
\end{enumerate}


The visibility notions from Definition \ref{def:vis} relate to the well-known P\'{o}lya's orchard problem (see \cite{Polya, Polya_book}): What is the minimal radius of trees (viewed as disks in $\R^2$), that stand at the integer points in a ball of radius $R$, for them to completely block the visibility of the origin, from the boundary of the ball? this problem was solved by Allen in \cite{Allen}, and some variants of it appears in \cite{HK, K}. One may also consider a maximal packing of unit balls in a ball of radius $R$, instead of balls at integer points, and ask for which $R$ (if any) there exists points which are not visible from the boundary? The existence of such an $R$ is known as Mitchell's dark forest conjecture, see \cite{Mitchell}. Mitchell's conjecture was proved in \cite{DJ}. Another related notion is the following. $Y\subseteq\R^d$ is called a \emph{dense forest} if every point $x\in \R^d$ is hidden, and for every $\eps$ there is a uniform upper bound $T(\eps)$ on the length of the line segments that are not $\eps$-close to $Y$. $T(\eps)$ is called the visibility function of $Y$. Questions regarding the existence of dense forests that are uniformly separated, or of bounded density, and bounds on the visibility functions of them, were studied in \cite{Adiceam, Alon, Bishop, SW}.


Our main results are the following. 

\begin{thm}\label{thm:Michael's_conj}
For every $\varepsilon>0$ and every relatively dense set $Y\subseteq\R^2$ there exist a number $T>0$ and 
a point $x\in\R^2$ such that for every $v\in\SS^1$ we have 
\[\dist(Y,\{x+tv\mid t\in[0,T]\})<\varepsilon.\] 
In particular, we have $\vis(Y;\varepsilon)\neq\R^2$.
\end{thm}



Theorem \ref{thm:Michael's_conj} generalizes the main result of Dumitrescu and Jiang from \cite{DJ}  
that settled Mitchell's dark forest conjecture. 

The following two theorems (Theorem \ref{thm:sublinear_growth_and_no_visibility}
and \ref{thm:small_growth_implies_full_visibility_2})
address the connection between the growth rate of a dis\-crete set $Y$  and the size of the set   $\vis(Y)$.
\smallskip

\begin{thm}\label{thm:sublinear_growth_and_no_visibility}
There exists a uniformly separated set $Y\subseteq\R^2$  such that\, $\vis(Y)=\varnothing$  and 
\[
\lim_{r\to\infty}\tfrac{G_Y(r)\log r}r\,=1.
\]
(In particular, the growth rate of such  $Y$  is sublinear,  $\lim_{r\to\infty}\limits \frac{G_{Y}(r)}r=0$,  and all points
in $\R^2$  are hidden for $Y$).
\end{thm}

\begin{thm}\label{thm:small_growth_implies_full_visibility_2}
Let  $Y\subseteq\R^d$  be a discrete set.  Then the implications \ (1)\,$\Rightarrow$(2)\,$\Rightarrow$(3) \
take place, where
\begin{enumerate}
\item $G_{Y}(r)<\frac{r^{d-1}}{\log^{1+\varepsilon} r}$,  
for some $\varepsilon>0$ and all large $r$. \\
\item  $\sum_{y\in Y\setminus\{{\bf0}\}}\limits\frac1{\norm{y}^{d-1}}<\infty$.\\
\item For every  $x\in\R^d$,  the relation  $x\in\vis(Y,v)$  holds for Lebesgue almost all  $v\in\SS^{d-1}$.
In particular,   $\vis(Y)=\R^d$.
\end{enumerate}
\end{thm}

It is easy to see that, for all $d\geq2$,  $\vis(\Z^d)=\R^d\!\setminus\!\Z^d$.  On the other hand,
the following theorem shows existence of large (density 1 and relatively dense) subsets  $Y\subseteq\Z^d$  
with no hidden points for $Y$.

\begin{thm}\label{thm:Every_point_is_visible}
	Let $d\ge2$. Then, for any $\varepsilon>0$  and  $M>1$,
	there exists a subset $Y\subseteq\Z^d$  which enjoys the following properties:
	\begin{enumerate}
   	   \item $\vis(Y)=\R^d$  (that is,  there are no hidden points for $Y$).
	   \item The growth rate of the complement set\,  $\tilde Y=\Z^d\!\setminus\! Y$  is at most linear;
	       moreover, $G_{\tilde Y}(r)=\#\big(\tilde Y\cap B({\bf0},r)\big)<\varepsilon r$,  for all  $r>0$.
	   \item  $Y$  is relatively dense in  $\Z^d$.
	   \item  The set  $\tilde Y$ is  $M$-separated.
        \end{enumerate}
\end{thm}

As an application to our approach we also prove a Ramsey type theorem, 
Theorem \ref{thm:multidim-Szemeredi-analog}, which is in the flavor of the multidimensional 
Szemer\'{e}di's theorem  (see \cite{FK, Zhao})  but is much easier to prove  (see Remark \ref{rem:ramsey}).


Given a discrete set $Y$ we say that \emph{almost every $y\in Y$} satisfies (some) property $(P)$   if 

\begin{equation}\label{eq:almostP}
\lim_{r\to\infty}\frac{\#\{y\in Y\cap B(r,{\bf0})| \ y \ \text{ does not have property }(P)\}}{G_Y(r)}=0.
\end{equation}

\begin{definition}\label{def:eps-realize_tree}
Let $Y\subseteq\R^2$ be discrete set, $\eps>0$, and $\Gamma$ a tree\footnote{An undirected, acyclic, connected graph.} embedded in the plane with vertices $V=\{x_0,\ldots,x_m\}$. Given $y_0\in Y$, we say that \emph{$(\Gamma,x_0)$ can be $\varepsilon$-realized from $y_0$ in $Y$} if there exists a function $f\colon V\to Y$ such that $f(x_0)=y_0$ and for every edge $\{x_i,x_j\}$ of $\Gamma$ there is an integer $k_{ij}\geq1$ such that 
\[\norm{(f(x_i)-f(x_j)) - k_{ij}(x_i-x_j)}<\varepsilon.\]  
\end{definition}

\begin{thm}\label{thm:multidim-Szemeredi-analog} 
Let $\eps>0$, $Y\subseteq\R^2$ a uniformly separated set with $\lim_{T\to\infty}\limits \frac{T}{G_Y(T)}=0$, $\Gamma=(V, E)$ a finite tree embedded in $\R^2$, and $x_0\in V$.
Then for almost every $y_0\in Y$ (in the sense of \eqref{eq:almostP}), $(\Gamma,x_0)$ can be $\varepsilon$-realized from $y_0$ in $Y$.
\end{thm}

Figure \ref{F:Trees} illustrates the statement of the theorem.
\begin{figure}[ht!]
\begin{center}
\includegraphics{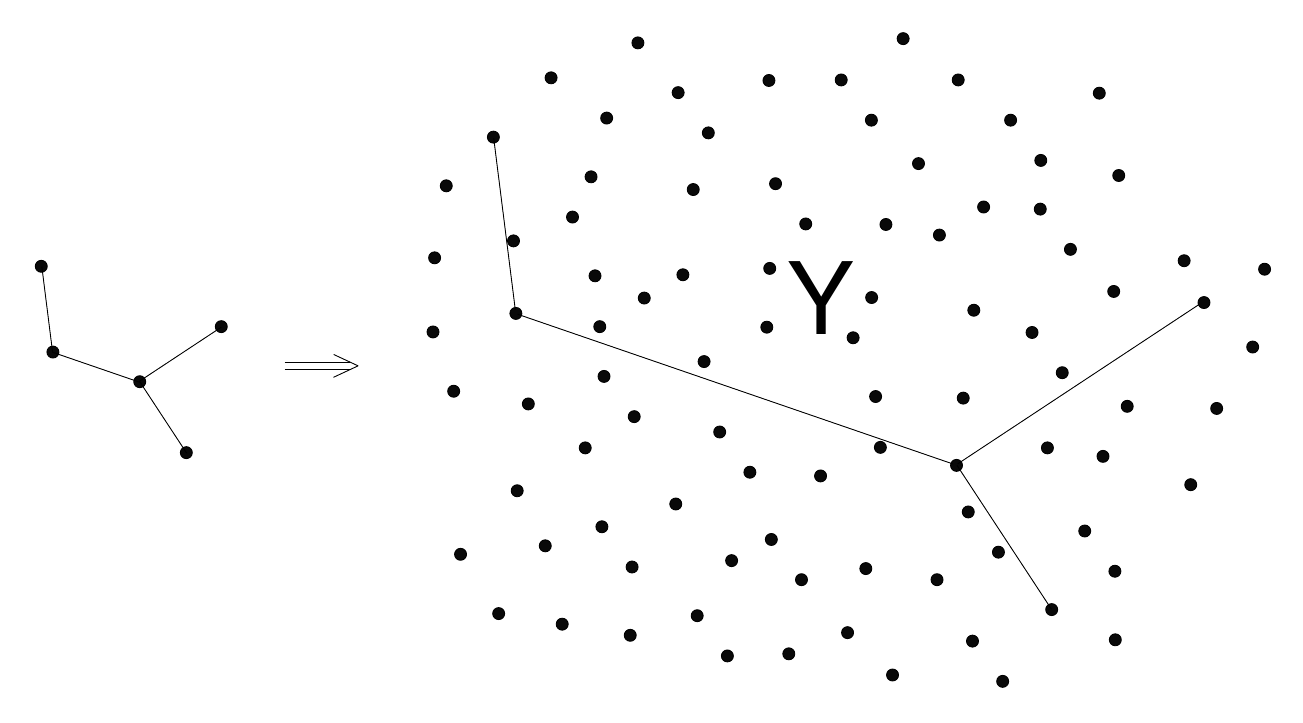}
\caption{Different edges may be stretched by different integer factors.}
\label{F:Trees}
\end{center}
\end{figure} 

\begin{remark}\label{rem:ramsey}
Observe that our assertion in Theorem \ref{thm:multidim-Szemeredi-analog} is weaker than the multidimensional Szemer\'{e}di's theorem in the sense that we allow different scalings for different edges of the tree, but here we assume no structure on $Y$ and hence our assumptions are much weaker as well. Other Ramsey type results of geometric nature can be found in \cite{D}.
\end{remark}

\subsection*{The structure of the paper} 
The proofs of Theorems \ref{thm:sublinear_growth_and_no_visibility}, \ref{thm:Every_point_is_visible},  and \ref{thm:multidim-Szemeredi-analog} are given in Sections \ref{sec:sublinear_growth_and_no_visibility}, \ref{sec:Every_point_is_visible} and \ref{sec:Realizing_Trees_in_Discrete_Sets} respectively. The proof of Theorem \ref{thm:small_growth_implies_full_visibility_2}
 is presented only in the special case of $d=2$ (Theorem \ref{thm:small_growth_implies_full_visibility_1} in 
 Section \ref{sec:full_vis}), the general case of $d\geq2$ is analoguous. Being more involved, the proof of Theorem \ref{thm:Michael's_conj} is only given in 
 Section \ref{sec:proof_of_Michael's_conj}. We conclude with open problems in Section \ref{sec:open_problems}.


\ignore{
\subsection*{Acknowledgments} 
The authors would like to thanks Barak Weiss for useful discussions and for Lemma \ref{lem:compactness_argument_for_eps-hidden_pts}.\\
\comyaar{Lemma \ref{lem:compactness_argument_for_eps-hidden_pts} (with this acknowledgment) can be erased if we don't find any way to exploit it.}
}

\section{Proof of Theorem \ref{thm:sublinear_growth_and_no_visibility}}\label{sec:sublinear_growth_and_no_visibility}
Consider the set 
\begin{subequations}\label{eq:y}
	\begin{equation}\label{eq:yshort}
	Y\df\{y_k\mid k\geq2\}\subseteq  \C, 
	\end{equation}
	where \vspace{-1mm}
	\begin{equation} \label{eq:ylong}
	y_k=r_ke^{i\phi_k}\in\C,\qquad \text{and} \qquad  \left\{
	\begin{array}{l} r_k\,=k\log k   \\
	\phi_k=\log^{1/2}(\log k) 
	\end{array}
	\right.
	\end{equation}
\end{subequations}
(with  $\R^2$ and $\C$ being identified). 

Theorem \ref{thm:sublinear_growth_and_no_visibility} is derived from the following
proposition.

\begin{prop}\label{prop:ex1Y}
	The set \ $Y$  in \eqref{eq:y}  provides an example for Theorem \ref{thm:sublinear_growth_and_no_visibility}.
	Thus $Y$  has no hidden points  (i.\,e., $\vis(Y)=\R^2$) and  
	\begin{equation}\label{eq:growthrate1}
	\lim_{r\to\infty}\frac{G_Y(r)\log r}r\,=1.
	\end{equation}
\end{prop}
\begin{proof}
	The growth rate \eqref{eq:growthrate1}  of  $Y$   is easily validated. 
	In order to prove that  $\vis(Y)=\R^2$,  we have to show
	that the distance between  the set $Y$  and any 
	ray  $L_{x,v}=\{x+tv\mid t\in[0,\infty)\}\subseteq\R^2$  is  $0$:  
	\begin{equation}\label{eq:distLY}
	\dist(L_{x,v},Y)=0\qquad  (\forall x\in \R^2, \ \forall v\in \SS^1).
	\end{equation}
	(see \eqref{eq:dist}  for the definition of  the distance function $\dist$).
	
	Fix $x\in \R^2$ and $v\in \SS^1$.
	Since the union  
	$U=\bigcup_{k\geq3}\limits\overline{y_k,y_{k+1}}$  \ 
	of the segments  $\overline{y_k,y_{k+1}}$ \
	forms an  expanding spiral in $\R^2$ (spinning counterclockwise), the set 
	\begin{equation}\label{eq:K}
	K=K(L_{x,v}):=\left\{k\geq 3 \mid  \overline{y_k,y_{k+1}} \cap L_{x,v}\neq\varnothing\right\}
	\end{equation}
	must be infinite.  We shall prove that in fact
	\begin{equation}\label{eq:distLyk}
	\lim_{{k\to\infty}\atop{k\in K}} \dist(L_{x,v}, y_k)=0.
	\end{equation}
	This would imply \eqref{eq:distLY}  and complete the proof of Proposition \ref{prop:ex1Y}.
	
	Observe the following three estimates (see \eqref{eq:ylong}):
	\begin{subequations}
		\begin{align}
		\phi'_k\df\phi_{k+1}-\phi_k= &  \ O\Big(\tfrac1{k\,\cdot\,\log k\,\cdot\,\log^{1/2}(\log k)}\Big), \label{eq:phipkO}\\ 
		r'_k\df\,r_{k+1}-r_k=& \  O(\log k) \label{eq:rpkO}  \\
		y_{k+1}-y_k=& \ O(\log k).  \label{eq:ypkO}
		\end{align}
	\end{subequations}
	The  first two, \eqref{eq:phipkO} and \eqref{eq:rpkO}, 
	are straightforward, and the third one easily follows:
	\begin{align*}
	|y_{k+1}-y_k|&=|r_{k+1}e^{i\phi_{k+1}}-r_ke^{i\phi_k}|\leq |r_{k+1}e^{i\phi_{k+1}}
	-r_ke^{i\phi_{k+1}}|\\ 
	&+r_k|e^{i\phi_{k+1}}-e^{i\phi_k}|= |r'_k|+r_k\,|e^{i\phi'_k}-1|=O(\log k)\\
	&+O(r_k\cdot\phi'_k)=O(\log k)+O\big(\tfrac1{\log^{1/2}(\log k)}\big)=O(\log k).
	\end{align*}
	
	For  $x\in\R^2$, denote by   $\phi'_{k,x}\in[0,\pi]$   the angle between the vectors \
	\mbox{$\overrightarrow{xy_k}=y_k-x$} and  $\overrightarrow{xy_{k+1}}=y_{k+1}-x$ (neither vector vanishes
	for a fixed $x$ and large $k$).  
	
	Note that in view of  \eqref{eq:ylong} and
	\eqref{eq:phipkO},  we have 
	\begin{equation}\label{ineq:phip0}
	\phi'_{k,{\bold 0}}=\phi_{k+1}-\phi_k=O\big(\tfrac1{k\,\cdot\,\log k\,\cdot\,\log^{1/2}(\log k)}\big)
	\end{equation}
	where ${\bold 0}=(0,0)$  stands for the origin in $\R^2$.
	
	Denote by  $S_{k,x}$ the area of the triangle $\triangle(x, y_k, y_{k+1})$, with vertices $x, y_k, y_{k+1}$.  
	Then
	\begin{equation}\label{eq:skx}
	S_{k,x}=\tfrac12\,|y_k-x|\cdot|y_{k+1}-x|\cdot  \sin\phi'_{k,x}
	\end{equation}
	and hence,  in view of \eqref{eq:ylong} and \eqref{ineq:phip0},  
	\begin{align}
	S_{k,{\bold0}}&=\tfrac12\,|y_k|\cdot|y_{k+1}|\cdot  \sin\phi'_{k,\bold0}\label{eq:sk0}\\
	&=O(k^2\cdot\log^2 k\cdot\phi'_{k,\bold0})=
	O\big(\tfrac{k\,\cdot\,\log k}{\log^{1/2}(\log k)}\big).\notag
	\end{align}
	
	Denote by  $\Re(z), \Im(z)\in\R$  its real and imaginary parts of $z\in\C$. 
	
	Since $x\in\C=\R^2$  is fixed,  the numbers $a=\Re(x)$, $b=\Im(x)$ are also fixed. Then  
	\[
	S_{k,x}= \frac{1}{2}\left|
	\det\!\begin{pmatrix}
	\Re(y_k)-a & \Re(y_{k+1}-y_k)\\
	\Im(y_k)-b & \Im(y_{k+1}-y_k)
	\end{pmatrix} 
	\right|
	\leq \big|S_{k,x}^{(1)}\big|+\big|S_{k,x}^{(2)}\big|
	\]
	where  \ $S_{k,x}^{(1)}= \frac{1}{2}\det\!
	\begin{pmatrix}
	\Re(y_k) & \Re(y_{k+1}-y_k)\\
	\Im(y_k) & \Im(y_{k+1}-y_k)
	\end{pmatrix} $
	\ and \ 
	$S_{k,x}^{(2)}=\det\!
	\begin{pmatrix}
	a & \Re(y_{k+1}-y_k)\\
	b & \Im(y_{k+1}-y_k)
	\end{pmatrix}.
	$
	
	In view of \eqref{eq:sk0}   and  \eqref{eq:ypkO},   we have \ 
	$\big|S_{k,x}^{(1)}\big|=S_{k,{\bold 0}}=O\!\left(\tfrac{k\,\cdot\,\log k}{\log^{1/2}(\log k)}\right)$ and  $\big|S_{k,x}^{(2)}\big|=O(|y_{k+1}-y_k|)=O(\log k)$
	(as  $x$  is fixed). Thus 
	\[
	S_{k,x}\leq \big|S_{k,x}^{(1)}\big|+\big|S_{k,x}^{(2)}\big|=
	O\big(\tfrac{k\,\cdot\,\log k}{\log^{1/2}(\log k)}\big).
	\]
	
	Since  $|y_k-x|^{-1}=O(k^{-1}\log^{-1}k)$,  it follows from \eqref{eq:skx} that
	\begin{equation}\label{est:sin}
	\phantom{W}\sin\phi'_{k,x}=O\big(|y_k-x|^{-1}\,|y_{k+1}-x|^{-1}\cdot S_{k,x}\big)=
	O\Big(\tfrac{1}{k\cdot\log k\,\cdot\,\log^{1/2}(\log k)}\Big).
	\end{equation}
	
	Now assume that  $k\in K$.  Then  the ray  $L_{x,v}$ intersects the segment  
	$[y_k,y_{k+1}]$. 
	Let $\psi_{k,x}$  be the angle 
	between the ray  $\overrightarrow{x,y_{k}}$ and the ray  $L_{x,v}$.
	This angle forms  a part of the angle between the vectors 
	\mbox{$\overrightarrow{xy_k}$} and  $\overrightarrow{xy_{k+1}}$, hence
	\[
	0\leq\psi_{k,x}\leq\phi'_{k,x}<\pi/2\qquad  (k\in K).
	\]
	Taking into account the estimate \eqref{est:sin},  we obtain
	\begin{align*}
	&\dist(L_{x,v},Y)\leq\dist(L_{x,v},y_k)=|y_k|\sin\psi_{k,x}\leq |y_k|\sin\phi'_{k,x}\\
	&\qquad=(k\cdot\log k)\cdot O\Big(\tfrac{1}{k\,\cdot\,\log k\,\cdot\,\log^{1/2}(\log k)}\Big)
	=O\big(\tfrac{1}{\log^{1/2}(\log k)}\big)  \quad   (k\in K).
	\end{align*}
	
	This proves \eqref{eq:distLY}  and completes the proof of Proposition \ref{prop:ex1Y}
	(and hence of Theorem~\ref{thm:sublinear_growth_and_no_visibility}).
\end{proof}

\section{Proof of Theorem \ref{thm:small_growth_implies_full_visibility_2}}\label{sec:full_vis}

The proof of Theorem \ref{thm:small_growth_implies_full_visibility_2}  is provided only for the case
of  $d=2$  (presented by Theorem~\ref{thm:small_growth_implies_full_visibility_1} below).
The general case of $d\geq2$  in Theorem \ref{thm:small_growth_implies_full_visibility_2} 
is handled in a similar way.

\begin{thm}\label{thm:small_growth_implies_full_visibility_1}
Let  $Y\subseteq\R^2$  be a discrete set. Then the following implications:  (1)\,$\Rightarrow$(2)\,$\Rightarrow$(3) \
hold where
\begin{enumerate}
\item $G_{Y}(r)<\frac r{\log^{1+\varepsilon} r}$  \
(for some $\varepsilon>0$ and all large $r$); \\[-1mm]
\item  $\sum_{y\in Y\setminus\{{\bf0}\}}\limits\frac1{\norm{y}}<\infty$;\\
\item For every  $x\in\R^2$,  the relation  $x\in\vis(Y,v)$  holds for Lebesgue almost all  $v\in\SS^1$.
In particular,   $\vis(Y)=\R^2$.
\end{enumerate}
\end{thm}

The proof of Theorem \ref{thm:small_growth_implies_full_visibility_1} is partitioned into two parts.
The implications (1)$\Rightarrow$(2)  and  (2)$\Rightarrow$(3)  are established by Propositions \ref{prop:sum_inv1}
and \ref{prop:sum_inv2}, respectively.

\begin{prop}\label{prop:sum_inv1}
Let  $Y\subseteq\R^2$  be a discrete subset such that \ $G_Y(r)<\frac r {\log^{1+\eps}r}$,   for some $\eps>0$  and all large $r$.  Then  $\sum_{y\in Y\setminus\{{\bf 0}\}}\limits\frac1{\norm{y}}<\infty$.
\end{prop}

\begin{proof}
For  $k\geq1$, denote \  $Y_k=\{y\in Y\!\mid \norm{y}\leq 2^k\}$  and
\[
 Z_k=Y_{k+1}\!\setminus\! Y_{k}=
\left\{y\in Y\,\big|\, 2^k<\norm{y}\leq 2^{k+1}\right\}.
\]
Then, for large $k$, we have
\[
\sum_{y\in Z_k}\limits \frac1{\norm{y}}\leq |Z_k|\cdot 2^{-k}\leq |Y_{k+1}|\cdot 2^{-k}\leq  \dfrac{2^{k+1}}{\log^{1+\eps}(2^{k+1})}\cdot 2^{-k}
=\frac {O(1)}{k^{1+\eps}}.
\]

It follows that   
\[
\sum_{y\in Y\setminus\{{\bf 0}\}}\limits \frac1{\norm{y}}\leq\!\! \sum_{y\in Y_1\setminus\{{\bf 0}\}}\limits \frac1{\norm{y}}+
\sum_{k\geq1}\limits\left(\sum_{y\in Z_k}\limits \frac1{\norm{y}}\right)<\infty,
\]
completing the proof of Proposition \ref{prop:sum_inv1}.
 \end{proof}

\begin{lem}\label{lem:sum_inv2}
Let  $Y\subseteq\R^2$  be a discrete subset such that  $\sum_{y\in Y\setminus\{{\bf 0}\}}\limits\frac1{\norm{y}}<\infty$  \ holds.  
Then, for Lebesgue almost all directions  $v\in \SS^1\subseteq\R^2$,  we have
${\bf 0}\in\vis(Y,v)$.  In particular,  ${\bf 0}\in\vis(Y)$.
\end{lem}
\begin{proof}[Proof of Lemma \ref{lem:sum_inv2}]
Recall that  
$L_{{\bf0},v}=\{vt\mid t\in[0,\infty)\}\subseteq\R^2$  
stands for the ray emanating from the origin in direction $v\in \SS^1$. 

Let  $Y'=Y\!\setminus\!\{\bf 0\}$ and $\eps>0$. Set
\begin{align*}\label{eq:defDY}
D_Y(\eps)&\df\{v\in\SS^1\mid \dist(L_{{\bf0},v}, Y')<\eps\}
\\&= \bigcup_{y\in Y'}\limits\{v\in\SS^1\!\mid\! \dist(L_{{\bf0},v}, y)<\eps\}.
\end{align*}
Then
\begin{equation}\label{eq:sum11}
\lambda(D_Y(\eps))\leq\sum_{y\in Y'}\lambda(\{v\in\SS^1\mid \dist(L_{{\bf0},v}, y)<\eps\})
\end{equation}
where $\lambda$  stands for  the Lebesgue measure on the unit circle $\SS^1\!\subseteq\R^2$, \mbox{$\lambda(\SS^1)\!=2\pi$}.

Now assume that  $0<\eps<\min_{y\in Y'}\limits\,\norm{y}$.  Then the inequalities  \ $0<\eps<\norm{y}$ \ hold
for every  $y\in Y'$,  and one  verifies that
\[
\lambda(\{v\in\SS^1\mid \dist(L_{{\bf0},v}, y)<\eps\})= 2\arcsin{\tfrac\eps{\norm{y}}} < \tfrac{\pi\eps}{\norm{y}},
\]
for every $y\in Y'$ (the inequality $2\arcsin t<\pi t$, for  $0<t<1$, is used).
 
By substituting the last inequality into \eqref{eq:sum11}, we derive that
$\lambda(D_Y(\eps))\leq \pi\eps\, c,$
where $c=c(Y)=\sum_{y\in Y'}\limits\,\dfrac1{\norm{y}}<\infty$.

Next consider the set  $D_Y=\{v\in\SS^1\mid \dist(L_{{\bf0},v}, Y')=0\}$.  Since  
\[
D_Y\subseteq\{v\in\SS^1\mid \dist(L_{{\bf0},v}, Y')<\eps\}=D_Y(\eps),
\]
and since  $\lim_{\eps\to0+}\limits\lambda(D_Y(\eps))=0$,  we conclude that  $\lambda(D_Y)=0$,
and hence  
\[
\lambda(\SS^1\!\setminus\! D_Y)=\lambda\{v\in\SS^1\mid \dist(L_{{\bf0},v}, Y')>0\}=1.
\]
This completes the proof of Lemma \ref{lem:sum_inv2}.
\end{proof}

\begin{prop}\label{prop:sum_inv2}
Let  $Y\subseteq\R^2$  be a discrete subset and let $x\in \R^2$  be an arbitrary point.  Assume that
$\sum_{y\in Y\!\setminus\!\{{\bf 0}\}}\limits\frac1{\norm{y}}<\infty$  \ holds.  
Then, for Lebesgue almost all directions  \mbox{$v\in \SS^1$},  we have
$x\in\vis(Y,v)$.  
\end{prop}

\begin{proof}
Let  $Z=Y-x=\{y-x\mid y\in Y\}$.  Observe the implication
\[
\sum_{y\in Y\setminus\{{\bf 0}\}}\limits\tfrac1{\norm{y}}<\infty \ \implies  \sum_{z\in Z\setminus\{{\bf 0}\}}\limits\tfrac1{\norm{z}}<\infty.
\]
By Lemma \ref{lem:sum_inv2}, for Lebesgue almost all directions  $v\in \SS^1$,  
we have  ${\bf 0}\in\vis(Z,v)$;  hence  $x\in\vis(Z+x,v)$.  Since  $Y=Z+x$,  we get $x\in\vis(Y,v)$.
\end{proof}

\section{Proof of Theorem \ref{thm:Every_point_is_visible}} \label{sec:Every_point_is_visible}

\ignore{
\commichael{ Do we need Lemma \ref{lem:ball_argument} below?}
We need the following lemma which is a straightforward consequence of 
Lemma~\ref{lem:one_edge_approx} (part (A))  presented in Section \ref{sec:Realizing_Trees_in_Discrete_Sets}.

\begin{lem}\label{lem:ball_argument}
Let $Y\subseteq\R^2$ be a discrete set that satisfies $\lim_{r\to\infty}\limits \frac{r}{G_Y(r)}=0$. Then, for every $\varepsilon>0$ and every $v\in\mathbb{S}^1$, almost every $y\in Y$ (in the sense of \eqref{eq:almostP}) satisfies $y\notin\vis(Y,v,\eps)$. 
\end{lem}


It follows that, under the conditions of Lemma \ref{lem:ball_argument},  for every $\varepsilon>0$ 
and every $n$-tuple of directions $v_1,\ldots,v_n\in\mathbb{S}^1$, almost every $y\in Y$ satisfies
\[
y\notin\bigcup_{k=1}^n\limits \vis(Y,v_k,\eps).
\]

\commichael{ Do we need Lemma \ref{lem:ball_argument} above?}

On the other hand, it is not difficult to construct a uniformly separated and relatively dense
set \mbox{$Y\subseteq\R^d$}, $d\ge 2$, such that every $y\in Y$ is visible, and even from the same direction. For instance, let $\{z_n\}$ be an enumeration of\, $\Z^d$, and let  $\delta_n$ be a sequence of distinct numbers in $(0,1/2)$  such that 
$\lim_{n\to\infty} \delta_n=0$, 
then one easily verifies that every $y\in Y\df\{z_n+(\delta_n,0,\ldots,0)\}$ is visible from any direction perpendicular
to $(1,0,\ldots,0)$. 
(Note that in this example  one can show that $\vis(Y) = \R^d\!\setminus\!\Z^d$). 

\vspace{1mm}\hrule\vspace{1mm}
}

In Theorem \ref{thm:Every_point_is_visible} we construct a large (density 1 and relatively dense) subset  $Y\subseteq\Z^d$  
with no hidden points for $Y$.

\begin{proof}[Proof of Theorem \ref{thm:Every_point_is_visible}]
For simplicity, the construction  is presented only for dimension \mbox{$d=2$}. The same idea works for general $d\ge 2$. 

\subsection*{Outline of construction}
We start with arbitrary ordering of the set  $\Z^2$  in a sequence $(z_k)_{k\geq1}$.

Then, we inductively construct an increasing sequence   $(m_k)_{k\geq1}$ of positive integers   
(the details are below, following  \eqref{eq:zzk}).  

Given $z_k$ and $m_k$,  the vectors $v_k\in \Z^2$  and the sets  $Y_k\subseteq\Z^2$ 
are determined  as follows: \vspace{-3mm}
\begin{align}
v_k&\df(m_k,1)\in\Z^2; \label{eq:vk} \\ 
Y_k&\df\{z_k+nv_k\mid n\geq1\}\subseteq\Z^2.\label{eq:Yk} 
\end{align}

Finally,  we define set $Y$ by setting
\begin{equation}\label{eq:Y}
\tilde Y\df\bigcup_{k\geq1}\nolimits Y_k;  \quad               Y\df\Z^2\setminus\!\tilde Y.
\end{equation}

We claim that every point  $z\in\R^2$  is visible for $Y$,  i.\,e.\  condition (1) of 
Theorem \ref{thm:Every_point_is_visible}  is satisfied  (regardless of the choice of  
integers $m_k$).

Indeed, if $z\notin \Z^2$, the claim is obvious  (then  $z$  must be visible in either a horizontal 
or a vertical direction).  
Otherwise  $z=z_k$  for some $k\geq1$,  and, since $Y\subseteq \Z^2\!\setminus\!Y_k$, we get 
\begin{equation}\label{eq:zzk}
z=z_k\in\vis(\Z^2\!\setminus\!Y_k, v_k)\subseteq\vis(\Z^2\!\setminus\!Y_k)\subseteq\vis(Y).
\end{equation}


\begin{figure}[ht!]
	\begin{center}
		\includegraphics{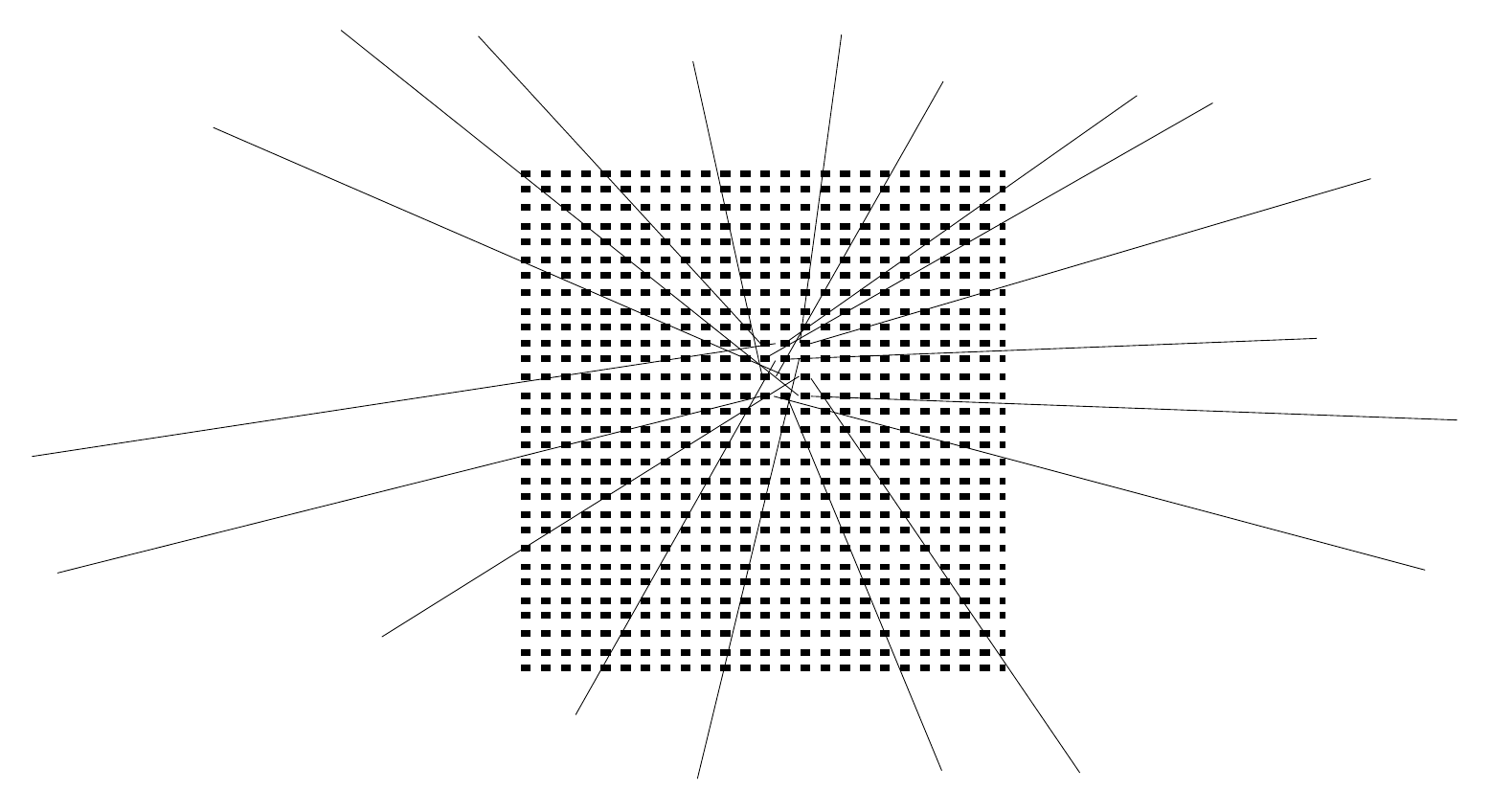}
		\caption{For each point $z_k$, a line of vision for $z_k$ is created by removing all the integer points on a particular ray, which is initiated in $z_k$.}
		\label{F:lines}
	\end{center}
\end{figure}

\subsection*{\bf Construction of a sequence ($m_k$)} We describe an inductive procedure for
selecting integers  $m_k$  to assure that the conditions (2), (3) and (4) of the theorem are met.  

One selects  an integer  $m_1>\max\{M, 4/\varepsilon, 2\!\norm{z_1}\}$  and proceeds by induction.

Assume that a strictly increasing $K$ terms long sequence of numbers  $(m_k)_{k=1}^{K}$  
has been already selected,  $K\!\geq1$.
Then the vectors  $v_k$ and the sets $Y_k$  are deter\-mined by  \eqref{eq:vk} and \eqref{eq:Yk}.
One easily verifies that for each $k=1, \ldots, K$
\begin{equation}\label{eq:distinctYs0}
\lim_{m\to+\infty}  \dist(Y_k, \{z_{k+1}+t(m,1)\mid t\in[1,\infty)\}=\infty
\end{equation}
(where $\dist$  is  the distance function defined in  \eqref{eq:dist}).

We select  $m_{K+1}$   large enough to satisfy the  inequalities
\begin{subequations}
\begin{align}
m_{K+1}>\max\{2^{K+1}\!/\varepsilon, 2\!\norm{z_{K+1}}, m_{K}\}\label{eq:mkineq}\\
\intertext{and}
 \dist(Y_k, Y_{K+1})>M \quad  (1\leq k\leq K) \label{eq:distinctYs}
\end{align}
\end{subequations}
where  
$Y_{K+1}=\{z_{K+1}+n(m_{K+1},1)\mid n\geq1\}
$
is set in accordance with \eqref{eq:Yk} and \eqref{eq:vk}
(note that the inequality \eqref{eq:distinctYs}  can be achieved because of \eqref{eq:distinctYs0}).

This completes the inductive construction of the sequence $(m_k)$.

\subsection*{Validation of condition (2)} 
Note that  \ $\norm{v_k}=\norm{(m_k,1)}>m_k>2\norm{z_k}$
for all $k\geq1$
 (see \eqref{eq:vk} and \eqref{eq:mkineq}).  It follows that, for all $n,k\in\N$,
\[
\norm{nv_k+z_k}\geq n\norm{v_k}-\norm{z_k} 
> (n-\tfrac12)\norm{v_k}\geq \tfrac n2\norm{v_k}> \tfrac {nm_k}2.  
\]

For any $k\geq1$,  in view of the definition of  $Y_k$  (see \eqref{eq:Yk}),  we obtain
\begin{align*}
G_{Y_k}(r)&=\#\{n\geq1\mid \norm{nv_k+z_k}< r\}\leq\\
&\leq \#\{n\geq1\mid \tfrac{nm_k}2< r\}<\tfrac{2r}{m_k}< 2r\varepsilon 2^{-(k+1)}=r\varepsilon 2^{-k},
\end{align*}
and, since  $\tilde Y=\bigcup_k Y_k$  (see \eqref{eq:Y}),   we conclude that
\[
G_{\tilde Y}(r)\leq \sum_{k\geq1} G_{Y_k}(r)<r\varepsilon\sum_{k\geq1} 2^{-k}=r\varepsilon,
\]
validating condition (2) of Theorem \ref{thm:Every_point_is_visible}.


\subsection*{\bf Validation of conditions (3) and (4)}  
To validate condition (4), we have to establish the implication \
$
(y_1,y_2\in\tilde Y, \ y_1\neq y_2)\implies  \dist(y_1,y_2)\geq M.
$

Since $y_1,y_2\in\tilde Y=\bigcup_k Y_k$, there are  $k_1,k_2\in \N$  such 
that  $y_1\in Y_{k_1}$, $y_2\in Y_{k_2}$.

If $k_1=k_2$, set  $k=k_1$;   then  $y_1,y_2\in\tilde Y_k=\{z_k+nv_k\mid n\geq1\}$,  and
(since $y_1\neq y_2$)  we obtain $\dist(y_1,y_2)\geq \norm{v_k}>m_k>M$.

And if $k_1\neq k_2$,  we may assume that  $k_1> k_2$,  and then
$
\dist(y_1,y_2)\geq\dist(Y_{k_1}, Y_{k_2})>M
$
in view the construction of $(m_k)$  (see \eqref{eq:distinctYs}). 

This validates condition (4). 
 
 In order to validate condition (3),  we show that every ball  $B$  of radius~$\sqrt2$ contains a point in $Y$.
(This claim holds only for $d=2$;  for $d\geq3$,  the required radius could be taken $\sqrt{d}$).

Let  $B=B(z,\sqrt2)$  where  \mbox{$z=(a,b)\in\R^2$}.  Then both points
$y_1=\left(\floor{a},\floor{b}\right)$ and $y_2=\left(\floor{a}+1,\floor{b}\right)$
lie in  $B\,\cap\,\Z^2$.  

Since $\norm{y_1-y_2}=1<M$,  we have $y_i\notin\tilde Y$
for at least one $i\in\{1,2\}$  (due to the already established condition (4)). 
But then $y_i\in Y=\Z^2\!\setminus\!\tilde Y$  
(see \eqref{eq:Y}), and hence $y_i\in B\cap Y$;  thus $B\cap Y\neq\emptyset$.
The proof of Theorem \ref{thm:Every_point_is_visible} is complete.\\
\end{proof}

\section{Proof of Theorem \ref{thm:multidim-Szemeredi-analog}}\label{sec:Realizing_Trees_in_Discrete_Sets}
We begin with the following lemma which is the key for the proof of 
Theorem~\ref{thm:multidim-Szemeredi-analog}.

\begin{lem}\label{lem:one_edge_approx}
Let $Y\subseteq\R^2$ be a discrete set such that  $\lim_{r\to\infty}\limits \frac{r}{G_Y(r)}=0$. 
Let  a non-zero vector ${\bf0}\neq v\in\R^2$ and an $\varepsilon>0$ be given.
Then:
\begin{itemize}
\item[(A)]  For almost every $z\in Y$  (in the sense of \eqref{eq:almostP}), 
one can find a point $w\in Y\!\setminus\!\{z\}$ and
 an integer $k\geq0$  such that
\begin{equation}\label{eq:one_edge_approx}
\norm{(z-w) - kv}<\varepsilon.
\end{equation} 

\item[(B)] (Under the additional assumption that  \ $Y$  is uniformly separated). \\ 
\mbox{} \hspace{1mm}  Given an integer $M\geq1$,  then, for almost every $z\in Y$,  one can find 
 a point $w\in Y\!\setminus\!\{z\}$ and an
integer  $k\geq M$ such that  \eqref{eq:one_edge_approx}  holds.
\end{itemize}
\end{lem}

Note that the assumption $v\neq{\bf0}$ in the above lemma is necessary.  (Indeed, take e.\,g.  $Y=\Z^2$ and  $\varepsilon=1/2$).

\begin{proof}[Proof of (A) in Lemma \ref{lem:one_edge_approx}]
Without loss of generality we may assume that $v=(1,0)$  and that  $\varepsilon<1$.  

Fix  $r>1$.  Divide the half-closed interval  $[-r,r)$ into 
\mbox{$N_1=\ceil{4 r/\varepsilon}$} half-closed subintervals $S_i\ (1\leq i\leq N_1)$
of equal length $d_1=\frac{2r}{N_1}\leq\eps/2$.

Divide the half-closed interval $[0,1)$ into $N_2=\ceil{2/\varepsilon}$ half-closed subintervals $I_j\ (1\leq j\leq N_2)$
of equal length  $d_2=\frac1{N_2}\leq\varepsilon/2$.

For  $z=(z_1, z_2)\in\R^2$,  denote by $\pi_j(z)=z_j\in\R$, $j=1,2$,  the coordinates 
of~$z$. For  $1\leq i\leq N_1$ and $1\leq j\leq N_2$, set  
\begin{equation}\label{eq:Yijr}
Y_{i,j}(r)\df\{y\in Y\cap B({\bf 0},r)\mid  \pi_2(y)\in S_i \ \text{ and } \  \{\pi_1(y)\}\in I_j\}
\end{equation}
where   $\{\pi_1(y)\}\in[0,1)$  stands for the fractional part of $\pi_1(y)$.

Next we prove the following implication:
\begin{align}\label{eq:implication}
z,w\in Y_{i,j}(r) \ &\implies \  \norm{(z-w) - kv}<\varepsilon, \\
&\text{\ where } \  k=\floor{\pi_1(z)-\pi_1(w)+1/2}\in\Z. \notag
\end{align}

Indeed, if $z,w\in Y_{i,j}(r)$  for some $i, j$ ($1\leq i\leq N_1$, $1\leq j\leq N_2$),  then
\begin{equation}\label{eq:dif:pi1}
|\{\pi_1(z)\}-\{\pi_1(w)\}|<|I_j|=d_2\leq \varepsilon/2
\end{equation}
and hence (since $\varepsilon<1$)
\[
|(\pi_1(z)-\pi_1(w))-k|<\varepsilon/2<1/2
\]
where  
\[
k=\floor{\pi_1(z)-\pi_1(w)+1/2}
\]
stands for the closest integer to  $\pi_1(z)-\pi_1(w)$.
We also have  
\begin{equation}\label{eq:dif:pi2}
|\pi_2(z)-\pi_2(w)|<|S_i|=d_1\leq \varepsilon/2,
\end{equation}
and hence 
\[
\norm{(z-w) - kv}=\norm{\big(\pi_1(z)-\pi_1(w)-k,\ \pi_2(z)-\pi_2(w)\big)}<
\sqrt2\cdot\varepsilon/2<\varepsilon,
\]
completing the proof of the implication \eqref{eq:implication}.

Denote by  $Y'$   the set of  $z\in Y$   such that for every  $w\in Y\!\setminus\!\{z\}$
and every integer  $k\geq0$  the inequality
\begin{equation}\label{eq:one_edge_approx_opp}
\norm{(z-w) - kv}\geq\varepsilon
\end{equation} 
holds  (cf. \eqref{eq:one_edge_approx}).
We claim that 
\begin{equation}\label{eq:singleton}
\#(Y_{i,j}(r)\cap Y')\leq1,  \ \text{ for all }  1\leq i\leq N_1, 1\leq j\leq N_2.
\end{equation}

Indeed, assume to the contrary that   $z,w\in Y_{i,j}(r)\cap Y'$ (for some  $i, j$), with $z\neq w$.
Then, in view of \eqref{eq:implication}, we have  
\[
\norm{(z-w) - kv}<\varepsilon,  \quad  \text{ with } \ k=\floor{\pi_1(z)-\pi_1(w)+1/2}.
\]
Assuming that  $\pi_1(z)\geq \pi_1(w)$  (otherwise renaming $z$ and $w$), we obtain
$k\geq0$.  This contradicts the assumption that  $z\in Y'$  (see \eqref{eq:one_edge_approx_opp}).

Since  $B({\bf0},r)\cap Y'=\bigcup_{i,j}  (Y_{i,j}(r)\cap Y')$,  we conclude, in view of \eqref{eq:singleton},  that  
\[
\#(B({\bf0},r)\cap Y')\leq \#((i,j))=N_1N_2=\ceil{4 r/\varepsilon}\cdot\ceil{2/\varepsilon},
\]
and hence \
$
\limsup_{r\to\infty}\limits\frac{\#(B({\bf0},r)\cap Y')}r\leq \ceil{4/\varepsilon}\cdot\ceil{2/\varepsilon}.
$
Finally, since  $\lim_{r\to\infty}\limits \frac{r}{G_Y(r)}=0$,  we derive
$
\lim_{r\to\infty}\limits\tfrac{\#(B({\bf0},r)\cap Y')}{G_Y(r)}=0,
$
completing the proof of (A) in Lemma \ref{lem:one_edge_approx}.
\end{proof}

\begin{proof}[Proof of (B) in Lemma \ref{lem:one_edge_approx}]
As in the proof of (A),  without loss of generality we assume that  $v=(0,1)$.  Since  $Y$  is uniformly separated,
there exists a  $\delta>0$  such that $\norm{y_1-y_2}\geq \delta$,  for all distinct  $y_1,y_2\in Y$.

We assume (as we may)  that  $\varepsilon<\delta<1$.

Fix  $r>1$.  Define  the integers $N_1, N_2$, the intervals $S_i, I_j$,  the numbers $d_1,  d_2$  
and  the sets $Y_{i,j}(r)$  (for $1\leq i\leq N_1, 1\leq j\leq N_2$)  just as in the proof of (A)  (see \eqref{eq:Yijr}).
We claim that 
\begin{equation}\label{eq:pi1sep}
|\pi_1(y_1)-\pi_1(y_2)|>\delta/2  \quad  \text{ (for distinct } y_1,y_2\in  Y_{i,j}(r)).
\end{equation}

Indeed, we have
\[
 (\pi_1(y_1)-\pi_1(y_2))^2+(\pi_2(y_1)-\pi_2(y_2))^2=\norm{y_2-y_1}^2>\delta^2
\]
and,  since  $\varepsilon<\delta$ and $|\pi_2(y_1)-\pi_2(y_2)|<\varepsilon/2$ \ (see  \eqref{eq:dif:pi2}),
we get
\[
 (\pi_1(y_1)-\pi_1(y_2))^2>\delta^2-\varepsilon^2/4>\delta^2-\delta^2/4>\delta^2/4,
\]
and \eqref{eq:pi1sep}  follows.

Denote by  $Y'_M$  the set of $z\in Y$  such that for every  $w\in Y\!\setminus\!\{z\}$ and every
$k\geq M$  the inequality   
\begin{equation}\label{eq:one_edge_approx_opp2}
\norm{(z-w) - kv}\geq\varepsilon
\end{equation} 
 holds.  
 
 Let  $N=\ceil{2M/\delta\,}$.  We claim that 
\begin{equation}\label{eq:boundset}
\#(Y_{i,j}(r)\cap Y'_M)\leq N,  \ \text{ for all }  1\leq i\leq N_1, 1\leq j\leq N_2.
\end{equation}
That is, no set $Y_{i,j}(r)\cap Y'_M$  contains more than  $N$  elements.

Assume to the contrary that, for some choice of  $i,j$,  we have  $N+1$  distinct elements
$y_1, y_2, \ldots, y_{N+1}$  lying in the same set $Y_{i,j}(r)\cap Y'_M$. We may assume
that these $N+1$ elements are arranged in such a way  that  $\pi_1(y_{p+1})-\pi_1(y_p)>\delta/2$,
for all  $p=1,2,\ldots, N$  (see \eqref{eq:pi1sep}). Then 
\[
\pi_1(y_{N+1})-\pi_1(y_1)>N\cdot\delta/2\geq M.
\]
In view of \eqref{eq:implication},  we obtain \ $\norm{(y_{N+1}- y_1)-kv}<\varepsilon$  \ where
\[
k=\floor{\pi_1(y_{N+1})-\pi_1(y_1)+1/2}\geq\floor{M+1/2}= M.
\]
This contradicts the assumption that  $y_{N+1}\in Y'_M$,
completing the proof of \eqref{eq:one_edge_approx_opp2}.

Since  $B({\bf0},r)\cap Y'_M=\bigcup_{i,j}\limits  (Y_{i,j}(r)\cap Y'_M)$, we conclude, in view of \eqref{eq:boundset},  that  
\[
\#(B({\bf0},r)\cap Y'_M)\leq M\cdot\#((i,j))=MN_1N_2=M\ceil{4 r/\varepsilon}\cdot\ceil{2/\varepsilon},
\]
and hence 
$
\limsup_{r\to\infty}\limits\frac{\#(B({\bf0},r)\cap Y'_M)}r\leq \frac{4M}\varepsilon\ceil{\frac2\varepsilon}.
$
\ Since  $\lim_{r\to\infty}\limits \frac{r}{G_Y(r)}=0$,  we derive
$
\lim_{r\to\infty}\limits\tfrac{\#(B({\bf0},r)\cap Y'_M)}{G_Y(r)}=0,
$
completing the proof of (B) in Lemma \ref{lem:one_edge_approx}.
\end{proof}
%

We completed the proofs of both parts of Lemma \ref{lem:one_edge_approx}.  Part (B) of this lemma, with $M=1$, is used 
in the proof of Theorem \ref{thm:multidim-Szemeredi-analog}. 

\begin{proof}[\bf Proof of Theorem \ref{thm:multidim-Szemeredi-analog}]
The proof is by induction. Assume the assertion for every tree with less than $m+1$ vertices, and let $\Gamma$ be a tree with $m+1$ vertices $V=\{x_0,\ldots, x_m\}$, embedded in the plane. Let $\Gamma'\df\Gamma\minus\{x_0\}$ (the graph obtained from $\Gamma$ by removing $x_0$ and all its adjacent edges). Let $c$ be the number of connected components of $\Gamma'$, then $\Gamma'$ is a disjoint union of $c$ trees $\Gamma_1,\ldots,\Gamma_c$, each with less than $m+1$ vertices. Denote by $x_{i_j}\in\Gamma_j$ the unique neighbor of $x_0$ in $\Gamma_j$, then for every $j\in\{1,\ldots,c\}$, by the induction hypothesis, for almost every $y\in Y$, $(\Gamma_j,x_{i_j})$ can be $\varepsilon$-realized from $y$ in $Y$. 

Let $Y_j\subseteq Y$ be the set of points $y\in Y$ for which $(\Gamma_j,x_{i_j})$ cannot be \mbox{$\varepsilon$-realized} from $y$. Then $Y'\df Y\minus(Y_1\cup\ldots\cup Y_c)$ still satisfies $\lim_{r\to\infty}\limits \tfrac{r}{G_{Y'}(r)}=0$, and, for every $y\in Y'$ and every $j\in\{1,\ldots,c\}$, the planar tree $(\Gamma_j,x_{i_j})$ can be $\varepsilon$-realized from $y$ in $Y$. For each $j$ consider the edge $\{x_0,x_{i_j}\}$ of $\Gamma$. By Lemma~\ref{lem:one_edge_approx} (part B), for almost every $y\in Y'$ there exists a positive integer $k_{i_j}$ and a point $z_{i_j}\in Y'$ such that
\[\norm{(y-z_{i_j}) - k_{i_j}(x_0-x_{i_j})}<\varepsilon.\]
Hence for almost every $y\in Y'$ there exist positive integers $k_{i_1}\ldots,k_{i_c}$ and points $z_{i_1},\ldots,z_{i_c}\in Y'$ such that for every $j\in\{1,\ldots,c\}$ we have 
\[\norm{(y-z_{i_j})-k_{i_j}(x_0-x_{i_j})}<\varepsilon,\]
and the assertion follows.  
\end{proof}

\section{Proof of Theorem \ref{thm:Michael's_conj}}\label{sec:proof_of_Michael's_conj}

Theorem \ref{thm:Michael's_conj} generalizes the main result of Dumitrescu and Jiang from their paper \cite{DJ}. Our proof is also a generalization of theirs. In \S\ref{subsec:proof_of_Micael's_conj} we repeat the main steps of the proof of Dumitrescu and Jiang, using similar terminology and parallel lemmas, adapted to our settings, and prove Theorem \ref{thm:Michael's_conj}. 

Note that since some of the parameters that are used in the proof are very large, and some are very small, some of our figures are drawn with wrong proportions.   

\subsection{Proof outline}
For every $z\in Y$ let $C_z=\partial B(z,\eps)$. We show that for many elements $z\in Y$ there are points on $C_z$ which are not $\eps$-visible. Like in \cite{DJ} we distinguish between two types of $\eps$-visible points on $C_z$; points $p\in C_z$ that are $\eps$-visible by a ray that is almost tangent to $C_z$ at $p$ are called \emph{tangentially visible}, and other $\eps$-visible points on $C_z$ are called \emph{frontally visible}. In Lemma \ref{lem:no_tangent_v} we show that every circle\footnote{Even every arc of every circle.} of radius $\eps$ contains points that are not tangentially visible. Then in Lemma \ref{lem:not_all_are_frontally_visible} we show that for a large enough $T$ only a fraction of the circles $C_z$ in $B({\bf 0},T)$, for $z\in Y$, contains points which are frontally visible. These two together imply that for a large enough $T$, some portion of the circles $C_z$ in $B({\bf 0},T)$, for $z\in Y$, contain points that are not $\eps$-visible. In particular, such points exist.


\subsection{Terminology}
Given a circle $C$ in the plane and $\sigma,\alpha\in [0,2\pi)$ we denote by $A(\sigma;\alpha)$ the arc of the circle $C$ that corresponds to the central angle that lies between $\sigma$ and $\sigma+\alpha$. The function $a:[0,2\pi)\to C$ maps an angle $\alpha$ to the point on $C$, which is the intersection of $C$ and the ray in direction $\alpha$ from the center of $C$. For two points $x,y\in\R^2$ we denote by $\overline{xy}$ the line segments that connects $x$ and $y$. 

\begin{definition}
Let $\eps,\delta>0$ and let $C\subseteq\R^2$ be a circle.
\begin{itemize}
\item
A point $p\in C$ is called \emph{$\delta$-tangentially-$\eps$-visible} ($\DTEV$) if $p$ is $\eps$-visible by a ray $L_{p,v}$ that satisfies:
\begin{itemize}
\item[(i)]
$L_{p,v}\cap C=\{p\}$ ($L_{p,v}$ intersects $C$ only at the tangent point).
\item[(ii)]
The angle between $L_{p,v}$ and the tangent to $C$ at $p$ is at most $\delta$.
\end{itemize}
\item
An arc of $C$ is called $\DTEV$ if {\bf every} point on that arc is $\DTEV$. 
\item If $p=a(0)$ is the point where the tangent to $C$ at $p$ is vertical, there are two directions in which a ray is almost tangent to $p$, and we distinguish between them in the following way. We say that a ray is \emph{pointing downwards} (respectively \emph{upwards}) to describe rays that point in these two directions, up to a small error. We say that $p$ is \emph{$\DTEV$ from below} (respectively \emph{$\DTEV$ from above}) if $p$ is $\DTEV$ by a ray pointing downwards (respectively upwards), up to an error angle $\delta$ at $p$ from the tangent to $p$. We adapt this terminology to other points $q=a(\alpha)$ on $C$ by rotating the plane so that $q=a(0)$. Note that this terminology will be used in the proof for points which are close to $a(0)$, where the rays truly point almost vertically downwards or almost vertically upwards. 
\item
A point $p\in C$ that is $\eps$-visible but not $\DTEV$ is called \emph{$\delta$-frontally-$\eps$-visible} ($\DFEV$). A ball $B$ is called $\DFEV$ if {\bf some} point on its boundary is $\DFEV$.
\end{itemize}     
\end{definition}

\subsection{Proof of Theorem \ref{thm:Michael's_conj}}\label{subsec:proof_of_Micael's_conj}
We begin with a lemma that asserts that for relatively dense sets $Y$, $\DTEV$ arcs does not exists, where $\delta=\delta(\varepsilon)$ is small enough. 

\begin{lem}\label{lem:no_tangent_v}
Let $Y\subseteq\R^2$ be an $R$-dense set. Then for every $\eps,\alpha>0$ there exists $\delta=\delta(\eps,R,\alpha)>0$ such that for every $x\in\R^2$, every arc of central angle $\alpha$ in $C\df\partial B(x,\eps)$ is not $\DTEV$.
\end{lem}

\begin{proof}
It suffices to prove the statement for $\alpha\in(0,\pi/6)$. Let $\eps>0, \alpha\in(0,\pi/6)$. To simplify notations, we may assume that $R$ is an integer multiple of $\eps$ (by replacing $R$ by some number in $[R,R+\eps)$). Set 
\begin{equation}\label{eq:delta}
N\df\frac{2R}{\eps},\qquad \beta\df\frac{\alpha}{4N}\qquad \text{ and }\qquad \delta\df \frac{\beta}{4N}=\frac{\alpha}{16N^2}. 
\end{equation}
Let $x\in\R^2$, and $C\df\partial B(x,\eps)$. Without loss of generality we prove the lemma for the arc $A\df A(0;\alpha)$. For contradiction, assume that $A$ is $\DTEV$.

Divide $A$ into $4N$ arcs, of equal length, with the points $q_i\df a\left(\frac{i\alpha}{4N}\right)$, for $i\in\{0,\ldots,4N\}$. Each of these sub-arcs has central angle $\beta$, and we denote it by $A_i\df A(q_i;\beta)$, for $i\in \{0,\ldots,4N-1\}$. We divide each $A_i$ into $4N$ arcs of equal length with the points $p_{i,j}\df a\left(\frac{i\alpha}{4N}+\frac{j\beta}{4N}\right)$, for $j\in\{0,\ldots,4N\}$. Consider the following two cases:\\
\underline{Case 1:} There exists an $i\in\{0,\ldots,4N-1\}$ such that all the points $p_{i,1},\ldots,p_{i,4N-1}$ are $\DTEV$ from below:\\  
For $j\in\{0,\ldots,4N-1\}$ let $L_j$ be the ray tangent to $C$ at $p_{i,j}$ that points downwards, $r_j$ a ray that indicates that $p_{i,j}$ is $\DTEV$ from below, and $L_j'$ the ray pointing downwards that intersects $C$ only at $p_{i,j}$ and that create an angle $\delta$ at $p_j$ between $L_j$ and $L_j'$. Let $L_{4N}$ be the ray tangent to $p_{i,4N}=p_{i+1,0}$ that points downwards. Denote by $z$ the intersection point of $L_0$ and $L_{4N}$ and let $a_0$ and $a_{4N}$ be two points on $L_0$ and $L_{4N}$ respectively such that the triangle with vertices $a_0, a_{4N}, z$ is the minimal isosceles triangle that contains a ball of radius $R$ (see Figure \ref{F:123} (a)). Since $\beta<\pi/6$, the legs of that triangle are indeed $\overline{za_0}$ and $\overline{za_{4N}}$, and the base is $I\df\overline{a_0a_{4N}}$.

\begin{figure}[ht!]
\begin{center}
\includegraphics[scale=0.4]{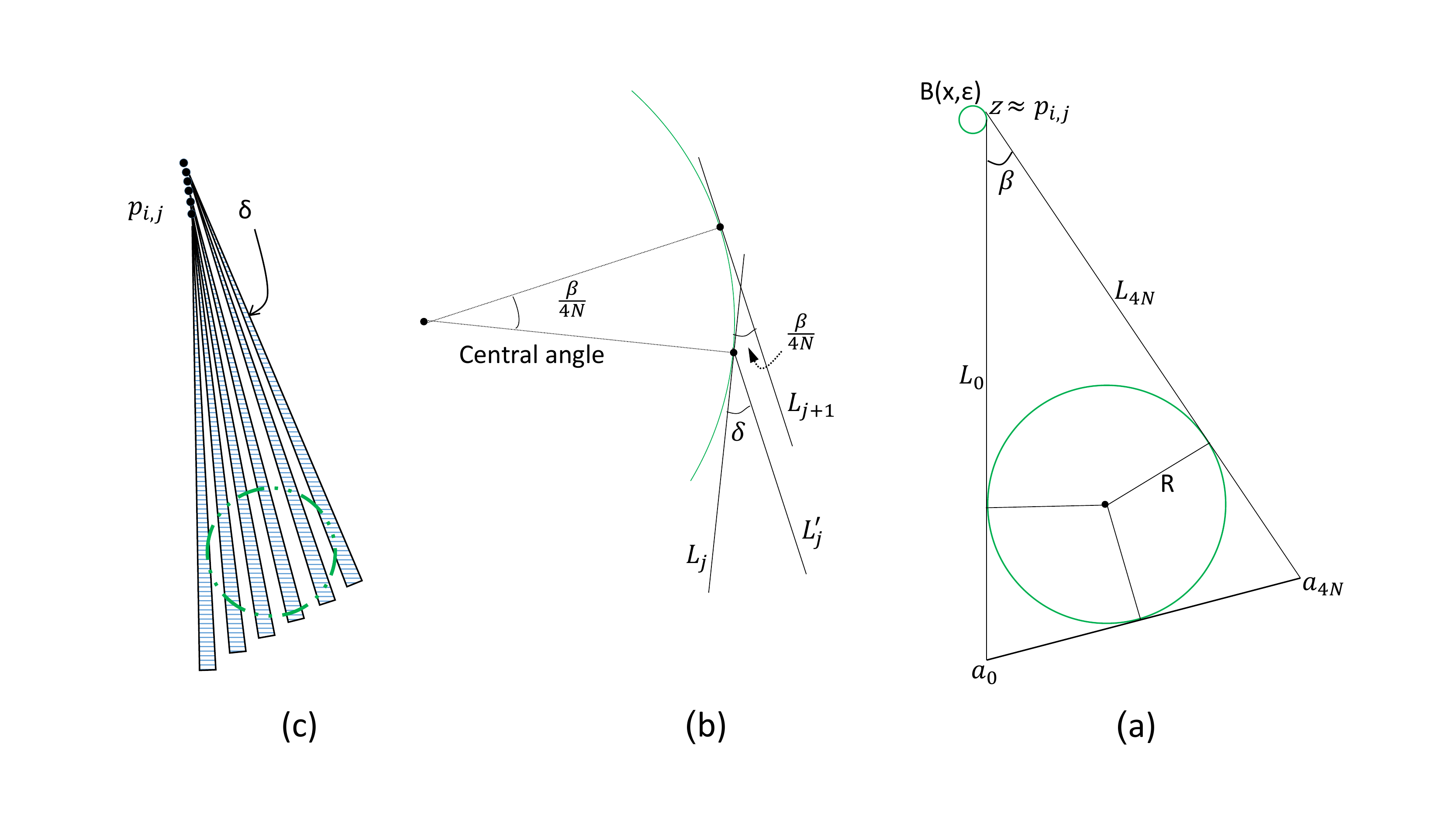}
\caption{}
\label{F:123}
\end{center}
\end{figure} 

For every $j\in\{0,\ldots,4N-1\}$ let $b_j$ be the intersection point of $r_j$ and $I$, and $a_j$ the intersection point of $L_j$ and $I$. Our next goal is to show that the points $\{b_0,\ldots,b_{4N-1}\}$ divide $I$ into segments of lengths less than $\eps$. This in turn implies that the rays $r_j$ divide $B$ in a way that every ball of radius $\eps$ that is centered in $B$ intersects at least one of the rays $r_j$ (see Figure \ref{F:123} (c)). Since $Y$ is $R$-dense, there exists some $y\in B\cap Y$, which contradicts the assumption that the points $p_{i,1},\ldots,p_{i,4N-1}$ are $\DTEV$ by the rays $r_1,\ldots,r_{4N-1}$.

Using elementary geometry (see Figure \ref{F:123} (a)) it is easy to show that for every $j\in\{0,\ldots,4N-1\}$ we have $\dist(p_{i,j},a_j)\le \frac{4R}{\beta}$. Since $\delta=\beta/4N$ the slope of the ray $L_{j+1}$ is equal to the slope of $L_j'$ (see Figure \ref{F:123} (b)). This implies that $b_j$ lies between $a_j$ and $a_{j+1}$ on $I$. In addition we have
\[
\dist(a_j,a_{j+1})\le 2\dist(p_{i,j},a_j)\sin\left(\frac{\beta}{8N}\right)\le 2\frac{4R}{\beta}\frac{\beta}{8N}=\frac{R}{N}=\frac{\eps}{2},
\]
which implies the assertion.  \\
\underline{Case 2:} For every $i\in\{0,\ldots,4N-1\}$ there is a $j\in\{0,\ldots,4N-1\}$ such that $p'_i\df p_{i,j}$ is $\DTEV$ from above:\\
We repeat the argument from case $1$ in a larger scale. For simplicity, we use the same notations. Here we denote by $L_i$, for $i\in\{0,\ldots,4N\}$, the ray tangent to $C$ at $q'_i=a\left(\frac{i\alpha}{4N}-\frac{\beta}{4N}\right)$ that points upwards, and by $r_i$, for $i\in\{0,\ldots,4N-1\}$, a ray that indicates that $p'_i$ is $\DTEV$ from above. Denote by $z$ the intersection point of $L_0$ and $L_{4N}$ and let $a_0$ and $a_{4N}$ be two points on $L_0$ and $L_{4N}$ respectively such that the triangle with vertices $a_0, a_{4N}, z$ is the minimal isosceles triangle that contains a ball of radius $R$. Since $\alpha<\pi/6$, the legs of that triangle are $\overline{za_0}$ and $\overline{za_{4N}}$, and the base is $I\df\overline{a_0a_{4N}}$. 

For $i\in\{0,\ldots,4N-1\}$ let $b_i$ be the intersection point of $r_i$ and $I$, and $a_i$ the intersection point of $L_i$ and $I$. Once again it is easy to verify that $\dist(q'_i,a_i)\le\frac{4R}{\alpha}$, and that $b_i$ lies between\footnote{This follows from the term $-\frac{\beta}{4N}$ that appears in the definition of the new points $q'_i$.} $a_i$ and $a_{i+1}$ on $I$, for every $i\in\{0,\ldots,4N-1\}$. In addition we have 
\[\dist(a_i,a_{i+1})\le 2\,\dist(q'_i,a_i)\sin\left(\frac{\alpha}{8N}\right)\le 2\,\frac{4R}{\alpha}\,\frac{\alpha}{8N}=\frac{R}{N}=\frac{\eps}{2},\]
which implies the assertion in a similar manner, and hence completes the proof of Lemma \ref{lem:no_tangent_v}.
\end{proof}

The proofs of the following two geometric lemmas are straightforward, and we leave them to the reader.

\begin{lem}\label{lem:sparse_intersections_on_the_boundary}
Let $T>\mu>0$. Suppose that $a_i=(x_i,y_i)\in\R^2, i\in\{1,\ldots,4\}$, satisfy (see Figure \ref{F:4})
\begin{equation}
\label{eq:positive_distance_between_points_at_dist_T}
x_1=x_2=5T, \ y_1\le y_2 \in \left[-\frac{\mu}{2},\frac{\mu}{2}\right], \ x_3=x_4, y_4-y_3\ge\mu,
\end{equation}
and $a_3,a_4\in B\df [-T,T]^2$. Let $\ell_1=\overline{a_1b},\  \ell_2=\overline{a_2c}$, where $b\neq c\in\{a_3,a_4\}$, and denote by $d_1,d_2$ the intersection point of $\ell_1,\ell_2$ respectively with $\partial B$. Then $\dist(d_1,d_2)\ge \frac{\mu}{3}$ (see Figure \ref{F:4}).  
\end{lem}

\ignore{
\begin{proof}
Consider the following two cases:\\
\underline{Case 1:} IF $b=a_3$, then $c=a_4$, and the line segments $\ell_1$ and $\ell_2$ has no intersection. Since $y_2-y_1\le\mu$ and $y_4-y_3\ge\mu$, it suffices to prove the assertion under the assumption that $y_1=y_2$, and in that case the proof is trivial.  
\\
\underline{Case 2:} IF $b=a_4$, then $c=a_3$, and the line segments $\ell_1$ and $\ell_2$ intersect at a point $z$. From $y_2-y_1\le\mu$ and $y_4-y_3\ge\mu$ one deduces that $\dist(z,B)\ge T$. Then using similar triangles we obtain
\[\dist(d_1,d_2)\ge \mu\cdot\frac{T}{3T}=\frac{\mu}{3}.\]
\end{proof}
}

\begin{figure}[ht!]
\begin{center}
\includegraphics[scale=0.3]{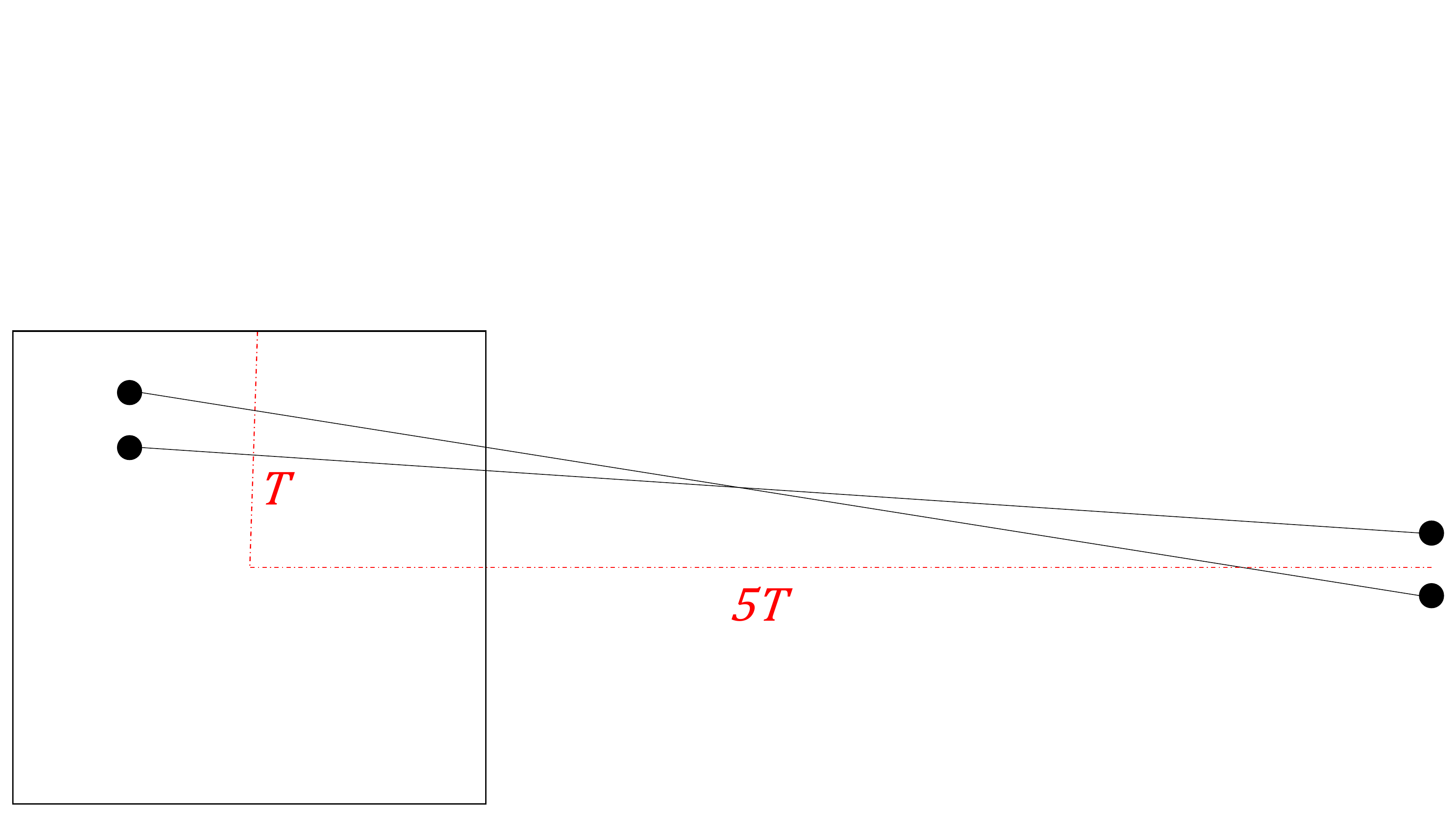}
\caption{}
\label{F:4}
\end{center}
\end{figure} 

\begin{lem}\label{lem:mu_def}
Let $\eps>\delta>0$. For $z\in\R^2$ if a ball $B(z,\eps)$ is $\DFEV$ by a ray $L$ then there is a point $a\in B(z,\eps)$, that lies on the continuation of $L$, with $\dist(a,\partial B(z,\eps))>\mu$, where
\begin{equation}\label{eq:mu_def}
\mu\df\frac{\eps\delta^2}{4}.
\end{equation} 
\end{lem}

\ignore{
\begin{proof}
Let $p\in\partial B(z,\eps)$ be a $\DFEV$ point by a ray $L$. The continuation of $L$ intersects a diameter of $B(z,\eps)$ perpendicularly at some point $a\in B(z,\eps)$. By elementary geometry we have
\[\dist(a,\partial B(z,\eps))\ge\eps(1-\cos(\delta))>\frac{\eps\delta^2}{4}=\mu.\]  
\end{proof}
}

For the proof of the next lemma we rely on the following proposition, see Lemma $6$ in \cite{DJ}.

\begin{prop}\label{prop:DJ's_Lemma}
Let $k,c,\eta>0$ and let $I$ be an interval of length $\absolute{I}$. Let $A\subseteq I$ be a finite set with at least $c\absolute{I}$ points, which are at least $\eta$ apart from each other. Set
\begin{equation}\label{eq:DJ_constants}
r\df\frac{k}{k-1}, j\df\ceil{\frac{\log\frac{2}{c\eta}}{\log r}}, \text{ and } Z_0=Z_0(k,c,\eta)\df 2\eta k^j.
\end{equation}
Then if $\absolute{I}\ge Z_0$ there exists some $x\ge 2\eta$, and a sub-interval $J\subseteq I$ of length $kx$ such that the subdivision of $J$ into $k$ equal sub-intervals $J_1,\ldots, J_k$ satisfies $J_i\cap A\neq\varnothing$ for every $i$. 
\end{prop}

\begin{lem}\label{lem:not_all_are_frontally_visible}
Let $\eps>0$, $Y\subseteq\R^2$ be an $R$-dense set such that $R$ is an integer multiple of $\eps$ and set $N=\frac{2R}{\eps}, \delta=\frac{\eps^2}{2^7\cdot R^2}$, and $C=\frac{1}{4R^2}$. Then for every 
\begin{equation}\label{eq:T_and_j}
T\ge \frac{\eps}{2}(4N)^j, \quad\text{ where }\quad  j=\ceil{\frac{33 + 10\log N}{\log(4N) - \log(4N-1)}},
\end{equation} 
the number of $z\in Y\cap B({\bf 0},T)$ such that $B(z,\eps)$ is $\DFEV$ is less than $CT^2$. 
\end{lem}

\begin{proof}

For contradiction, assume that for at least $CT^2$ points $z\in Y\cap B({\bf 0},T)$ the balls $B_z= B(z,\eps)$ are $\DFEV$. Each of these $CT^2$ balls has a point $p_z$ on its boundary and a ray $L_z$, initiated at $p_z$, indicating that $B_z$ is $\DFEV$. Denote by $\mathcal{R}$ the set of these rays, $L_z$. 

Set 
\begin{equation}\label{eq:mu}
\mu\df\frac{\eps\delta^2}{4}=\frac{\eps^5}{2^{16}R^4},\quad\text{and}\quad M\df\floor{\frac{32 T}{\mu}}.
\end{equation}

Consider the larger ball $B({\bf 0},5T)$ and place $M$ equally spaced\footnote{Note that $32T>10\pi T$, which is the length of the diameter of $B({\bf 0},5T)$.} points $p_0,\ldots,p_{M-1}$ on $\partial B({\bf 0},5T)$. The tangents to $B({\bf 0},5T)$ through the points $p_j$ form a regular $M$-gon that $B({\bf 0},5T)$ is inscribed in. Denote by $I_j$ the edge of that $M$-gon that contains $p_j$. Observe that the length of each segment $I_j$ is at most $\mu$. By the pigeonhole principle there exists some $j$ such that at least
\begin{equation}\label{eq:pigeonhole}
\frac{CT^2}{M} \ge \frac{CT^2}{\frac{32 T}{\mu}} = \frac{C\mu T}{32}
\end{equation}    
rays from $\mathcal{R}$ intersects $I_j$. We rotate the whole plane about the origin so that the segment $I_j$ is vertical, and denote by $\mathcal{R}_j\subseteq\mathcal{R}$ the subset of rays of $\mathcal{R}$ that intersects $I_j$. Note that all of these rays intersect the vertical line segment $I\df\{T\}\times[-T,T]$ of length $2T$ (see Figure \ref{F:B}). Let $A'\subseteq I$ be the set of these intersection points. 

Recall that each ray $L_z$ in $\mathcal{R}_j$ is initiated from a point $p_z\in\partial B_z$, for some $z\in Y$, such that $p_z$ is $\DFEV$ by $L_z$. So by our choice of $\mu$ in (\ref{eq:mu}) and by Lemma \ref{lem:mu_def} there is a point $a_z\in B_z$ with $\dist(a_z,\partial B_z)\ge\mu$. This implies that the requirements in (\ref{eq:positive_distance_between_points_at_dist_T}) are satisfied, and we can apply Lemma \ref{lem:sparse_intersections_on_the_boundary} for any such pair of rays, connecting points of the form $a_z$ to $I_j$ (see Figure \ref{F:B}). This in turn implies that the points of $A'$ are at least $\mu/3$ apart from each other, and in particular no two rays of $\mathcal{R}_j$ intersect $I$ at the same point. We pick a subset $A\subseteq A'$ such that any two points in $A$ are at least $\eps/2$ apart. This is done by ordering the elements of $A'$ and pick every $\ceil{\frac{3\eps}{2\mu}}$ point in that order. Thus, using (\ref{eq:pigeonhole}), we obtain that
\begin{equation}\label{eq:c_def}
\frac{\# A}{\absolute{I}}\ge
\frac{\# A'}{\ceil{\frac{3\eps}{2\mu}}\absolute{I}}\ge
\frac{\frac{C\mu T}{32}}{\ceil{\frac{3\eps}{2\mu}}\cdot2T}\ge
\frac{C\mu^2}{2^7\eps}= 
\frac{\mu^2}{2^{9}\eps R^2}
\df c.
\end{equation}
We apply Proposition \ref{prop:DJ's_Lemma} with $c$ as in (\ref{eq:c_def}), $k=4N$, and $\eta=\eps/2$. In view of (\ref{eq:mu}) and (\ref{eq:c_def}) we obtain
\[c\eta = c\eps/2 = \frac{\mu^2\eps}{2^{10}\eps R^2} = \frac{\eps^{10}}{2^{42}R^{10}}
\quad\Longrightarrow\quad 
\log\frac{2}{c\eta} = 33 + 10\log\frac{2R}{\eps} = 33 + 10\log N.\]
Therefore the constant $j$ at (\ref{eq:DJ_constants}) is
\[j = \ceil{\frac{33 + 10\log N}{\log\frac{4N}{4N-1}}} = \ceil{\frac{33 + 10\log N}{\log(4N) - \log(4N-1)}} 
.\]
Then the constant $Z_0$ at (\ref{eq:DJ_constants}) is
\[Z_0=2\eta k^j = \eps(4N)^j .\]
Thus, given the assumption on $T$ in \eqref{eq:T_and_j}, we have $\absolute{I}=2T\ge Z_0$. Applying Proposition \ref{prop:DJ's_Lemma} we obtain an $x\ge 2\eta=\eps$, and a sub-interval $J\subseteq I$ of length $4Nx\ge 4R$ such that the subdivision of $J$ into $4N$ equal sub-intervals $J_1,\ldots, J_{4N}$ satisfies $J_i\cap A\neq\varnothing$ for every $i\in\{1,\ldots,4N\}$. Let $L_1,\ldots,L_{4N}\in\mathcal{R}_j$ be the rays that correspond to those $4N$ points of $A$. Let $\Omega$ be the convex hull of $I_j\cup J$, then $\Omega$ clearly contains balls of radius $R$. Let $B\subseteq\Omega$ be the ball of radius $R$ that is tangent to the line segments that bound $\Omega$ from above and below (see Figure \ref{F:B}). Then $B\cap Y\neq\varnothing$ and every point $p\in B\cap Y$ is within distance at most $\eps/2$ from at least one of the rays $L_1,\ldots,L_{4N}$, contradicting our assumption on the rays in $\mathcal{R}$.  

\begin{figure}[ht!]
	\begin{center}
		\includegraphics[scale=0.5]{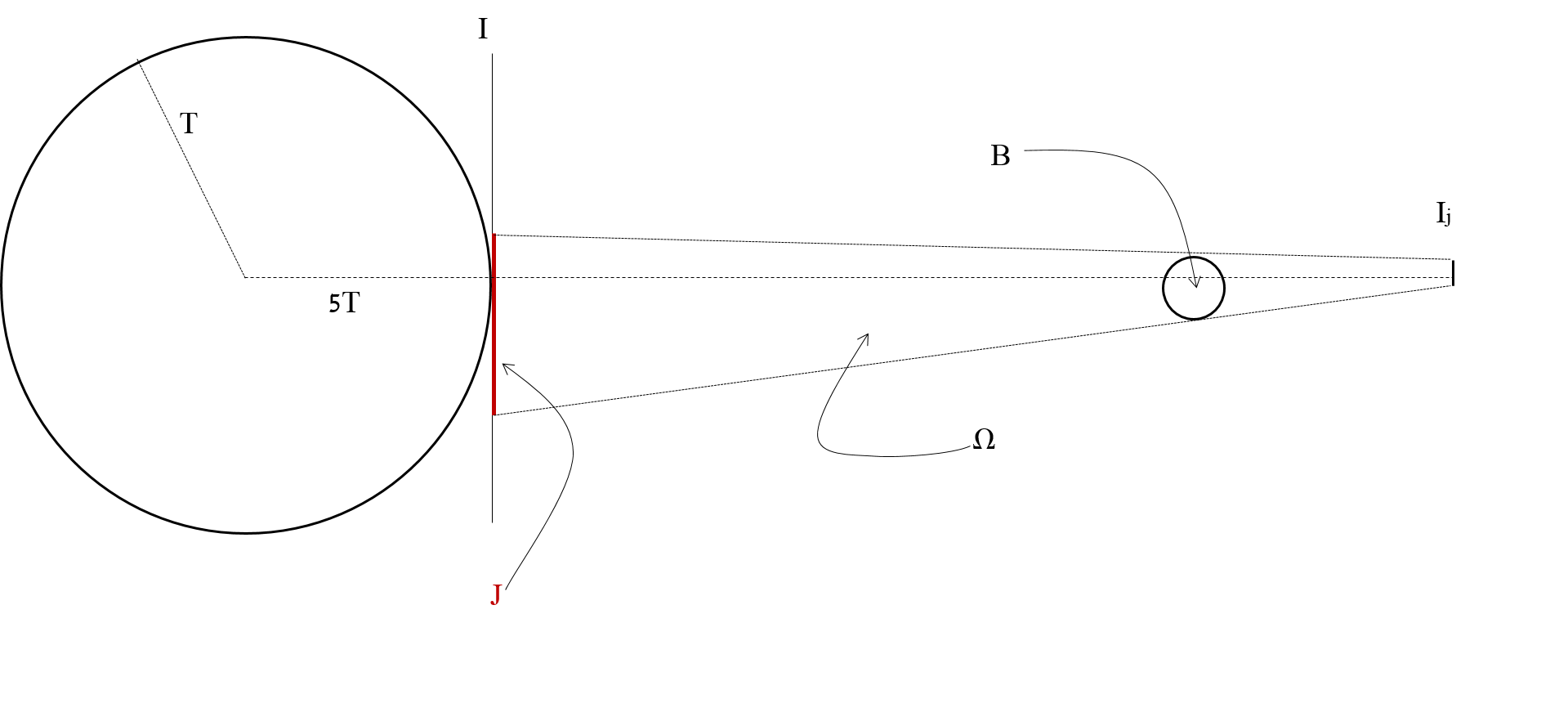}
		\caption{}
		\label{F:B}
	\end{center}
\end{figure} 

\end{proof}

\begin{proof}[\bf Proof of Theorem \ref{thm:Michael's_conj}]
Let $\eps>0$ and let $Y\subseteq\R^2$ be an $R$-dense set. By slightly increasing $R$ we may assume that $R$ is an integer multiple of $\eps$. Pick $T>0$ as in (\ref{eq:T_and_j}) and consider the ball $B\df B({\bf 0},T)$. It is easy to verify that $B$ contains at least $\frac{1}{2R^2}T^2$ disjoint balls of radius $R$, each one contains at least one element of $Y$. Denote by $Y' = Y \cap B({\bf 0},T)$, then we have established that $\# Y' \ge \frac{1}{2R^2}T^2$. 

Let $\delta=\frac{\eps^2}{2^7 R^2}$, as in Lemma \ref{lem:not_all_are_frontally_visible}. Note that this $\delta$ is consistent with our choice of $\delta$ at (\ref{eq:delta}) in Lemma \ref{lem:no_tangent_v}, for $\alpha = \frac{1}{2} < \frac{\pi}{6}$. So by Lemma \ref{lem:no_tangent_v}, in particular, every ball $B(z,\eps)$, for $z\in Y'$, is not $\DTEV$. Applying Lemma \ref{lem:not_all_are_frontally_visible} we obtain that for at most $\frac{1}{4R^2}T^2$ elements $z\in Y'$ the ball $B(z,\eps)$ is $\DFEV$. That leaves at least $\frac{1}{4R^2}T^2 = \frac{1}{2R^2}T^2-\frac{1}{4R^2}T^2$ many elements $z$ of $Y'$ for which the ball $B(z,\eps)$ is not $\DTEV$ and not $\DFEV$. Namely there are at least $\frac{1}{4R^2}T^2$ points, on the boundaries of these balls, which are not $\eps$-visible.  
\end{proof}

\section{Some Open Problems}\label{sec:open_problems}
We close with an open problem discussion.

\subsection{Generalizations to higher dimensions:} Most of our proofs rely on geometric ideas that, in some places, are hard to generalize. For the proof of Theorem \ref{thm:Michael's_conj}; Lemma \ref{lem:no_tangent_v} can be generalized to $d>2$ with a relatively small effort, simply because the idea that Figure \ref{F:123} (c) describes can be implemented in high dimension as well. On the other hand, Lemma \ref{lem:not_all_are_frontally_visible} may be more difficult to generalize, mostly because it relies on Proposition \ref{prop:DJ's_Lemma}. 

It would also be nice to generalize Theorem \ref{thm:sublinear_growth_and_no_visibility} to $\R^d$, $d>2$. 

\subsection{The dual-Danzer problem:} 
The following problem is known as the \emph{dual-Danzer problem} (see \cite[p. 285]{BC}):

\begin{minipage}{130mm}
\quad Suppose that $Y\subseteq\R^d$ $(d\ge 2)$ satisfies $\limsup_{r\to\infty}\limits\frac{G_Y(r)}{r^d}>0$. Is it true that for every $n\in\N$ there is a convex set $K\subseteq\R^d$, of volume 1, such that $\#(K\cap Y)\ge n$?
\end{minipage}
\medskip

\noindent As a step toward a negative answer to the dual-Danzer problem, one may consider ther followng weaker version: 

\begin{question}\label{que:arbmany}
Given a discrete and relatively dense set $Y$, is there a sequence of convex sets $K_n$, of some fixed volume, so that $\#(K_n\cap Y)\ge n$?
\end{question}

\ignore{
\comyaar{I am not happy with this sentence.} Observe that Theorem \ref{thm:Michael's_conj} proves a general property of relatively dense sets in the plane, and hence may be found useful for a future solution to the dual-Danzer problem.  
\commichael{I would prefer not to suggest that Theorem 2 may be useful in addressing  Danzer problem.
We may mention it as a difficult problem concerning relatively dense sets in the plane. It should be discussed.
Question \ref{que:arbmany}  is somewhat ambiguos and the corresponding claim seems to be a stronger, not a weaker version.}
}

\subsection{More related questions:} 
In the spirit of the problems discussed in this paper, we suggest the following directions of study.


\begin{question}
Observe that if $Y\subseteq\R^d$ is a lattice then $\vis(Y)=\R^d\!\setminus\! Y$. Which subsets of $\R^d$ are characterized by this property?   
\end{question}

One can consider questions concerning the set of direction from which the points are visible. 
Note that this set of directions cannot be dense in any ball of $\mathbb{S}^{d-1}$, because then the union of the $\eps$-neighborhood of the rays will cover arbitrarily large balls in $\R^d$, contradicting $Y$ being relatively dense. On the other hand, the following proposition is straightforward. 

\begin{prop}
	Given $\eps>0$ and $x\in\R^2$, there is a relatively dense set $Y\subseteq\R^2$ such that $x\in\vis(Y,v,\eps)$ for infinitely many directions $v\in\mathbb{S}^1$.  
\end{prop}
\begin{proof}
	Construct $Y$ with respect to the given point $x$, so that for the sequence of directions $v_n=\frac{\pi}{2^n}$ the relation $x\in\vis(Y,v_n,\eps)$ will hold for every $n\in\N$. The rate of convergence of the sequence implies that the union of the $\eps$-neighborhood of the rays in these directions does not cover a ball of radius greater than $2\eps$. Hence we can place a point of $Y$ in any ball of radius $R>2\eps$ without violating those directions of visibility.  
\end{proof}

\begin{question}
	Suppose that $\eps>0$ and $Y\subseteq\R^2$ is relatively dense such that for every $x\in\R^2$ the set of directions $v$ for which $x\in\vis(Y,v,\eps)$ is finite. Does this imply that $\vis(Y;\eps)\neq\R^2$, or even $Y\nsubseteq\vis(Y;\eps)$ (namely, that for some $x$ the above set of directions is empty)? 
\end{question}


\ignore{
With relation to those questions, we mention the following observation that was made by Barak Weiss.
\commichael{Not sure needs to be included}
\begin{prop}\label{lem:compactness_argument_for_eps-hidden_pts}
	Let $Y\subseteq\R^d$ be a discrete set $(d\ge 2)$. Then for every $\eps>0$ and $x\in\R^d$ if $x\notin \vis(Y;\eps)$ then there is a finite set $Y_x\subseteq Y$ such that $x\notin \vis(Y_x;\eps)$. 
\end{prop}

\begin{proof}
	Suppose that $x\notin \vis(Y;\eps)$, then for every $v\in \mathbb{S}^{d-1}$ there is some $y_v\in Y\!\setminus\!\{x\}$ in the open $\eps$-neighborhood of the ray $L_{x,v}$. Then there is an open neighborhood $v\in U_v\subseteq\mathbb{S}^{d-1}$ such that for every $u\in U_v$ the $\eps$-neighborhood of the ray $L_{x,u}$ also contains $y_v$. The collection $\{U_v\}_{v\in\mathbb{S}^{d-1}}$ forms an open cover of $\mathbb{S}^{d-1}$, which by compactness has a finite sub-cover. Let $U_{v_1},\ldots,U_{v_k}$ be such a finite cover, and let $y_{v_1},\ldots,y_{v_k}\in Y$ be the above points corresponding to these neighborhoods. Then the set $Y_x\df \{y_{v_1},\ldots,y_{v_k}\}$ is as required. 
\end{proof}
}

\end{document}